\documentclass[12pt]{amsart}
\usepackage[T1]{fontenc}
\usepackage[latin1]{inputenc}
\usepackage{typearea}
\usepackage{geometry}
\usepackage{ulem}
\usepackage{amsmath}
\usepackage{amssymb}
\usepackage{latexsym}
\usepackage{enumerate}
\usepackage{amsthm}
\usepackage[all]{xy}
\usepackage{hhline}
\usepackage{epsf} 
\usepackage{cite}
\newtheorem{theorem}{{\sc Theorem}}[section]
\newtheorem{cor}[theorem]{{\sc Corollary}}

\newtheorem{lemma}[theorem]{{\sc Lemma}}
\newtheorem{prop}[theorem]{{\sc Proposition}}

\theoremstyle{remark}
\newtheorem{remark}[theorem]{{\sc Remark}}

\theoremstyle{definition}

\newtheorem{example}[theorem]{\sc example}
\newtheorem{cond}[theorem]{\sc condition}

\newcommand{\R}{\mathbb{R} }
\newcommand{\N}{\mathbb{N} }

\newcommand{\A}{\mathcal{A}}

\newcommand{\F}{\mathcal{F}}

\newcommand{\D}{\mathcal{D}}

\newcommand{\calH}{\mathcal{H}}
\newcommand{\calL}{\mathcal{L}}

\newcommand{\la}{\lambda}
\newcommand{\al}{\alpha}

\newcommand{\ga}{\gamma}

\newcommand{\Om}{\Omega}

\newcommand{\abquer}{\overline{(a,b)}}
\newcommand{\htilde}{\tilde{h}}

\newcommand{\Ztilde}{\tilde{Z}}

\providecommand{\abs}[1]{\lvert #1\rvert}

\providecommand{\fnorm}[1]{\lVert #1\rVert_\infty}

\DeclareMathOperator{\Exp}{Exp}

\DeclareMathOperator{\sign}{sign}

\renewcommand{\phi}{\varphi}
\renewcommand{\epsilon}{\varepsilon}
\newcommand{\eps}{\varepsilon}
\renewcommand{\rho}{\varrho}

\begin{document}
\title[Stein's method for the Beta distribution]{Stein's method of exchangeable pairs for the Beta distribution and generalizations}
\author{Christian D\"obler}
\thanks{Universit\'{e} du Luxembourg, Unit\'{e} de Recherche en Math\'{e}matiques \\
christian.doebler@uni.lu }
\begin{abstract}
We propose a new version of Stein's method of exchangeable pairs, which, given a suitable exchangeable pair $(W,W')$ of real-valued random variables, suggests the approximation of the law of $W$ by a suitable absolutely continuous distribution. This distribution is characterized by a first order linear differential Stein operator, whose coefficients $\gamma$ and $\eta$ are motivated by two regression properties satisfied by the pair $(W,W')$. Furthermore, the general theory of Stein's method for such an absolutely continuous distribution is developed and a general characterization result as well as general bounds on the solution to the Stein equation are given. This abstract approach is a certain extension of the theory developed in the papers \cite{ChSh} and \cite{EiLo10}, which only consider the framework of the density approach, i.e. $\eta\equiv1$. As an illustration of our technique we prove a general plug-in result, which bounds a certain distance of the distribution of a given random variable $W$ to a Beta distribution in 
terms of a given exchangeable pair $(W,W')$ and provide new bounds on the solution to the Stein equation for the Beta distribution, which complement the existing bounds from \cite{GolRei13}. The abstract plug-in result is then applied to derive bounds of order $n^{-1}$ for the distance between the distribution of the relative number of drawn red balls after $n$ drawings in a P\'{o}lya urn model and the limiting Beta distribution measured by a certain class of smooth test functions.
\end{abstract}

\maketitle

\section{Introduction}\label{Intro}
Since its introduction in \cite{St72} in 1972 Stein's method has become a famous and useful tool for proving distributional convergence. One of its main advantages over other techniques is that it automatically yields concrete error bounds on various distributional distances. Being first only developed for 
normal approximation it was observed by several authors that Stein's idea of linking a characterizing operator for the target distribution to a differential equation, the \textit{Stein equation}, carries over to many other absolutely continuous and discrete distributions, where, in the discrete case, the differential equation 
has to be replaced by a suitable difference equation. Among those other distributions, to which Stein's method has been successfully extended, are the Poisson distribution (see e.g. \cite{Ch75}, \cite{AGG89} or \cite{BHJ}), the Gamma distribution (see \cite{Luk} or \cite{Rei05}), the exponential distribution 
(see e.g. \cite{CFR11}, \cite{PekRol11} and \cite{FulRos13}), the Laplace distribution \cite{PiRen12} and, more generally, the class of Variance-Gamma distributions \cite{Gau14}. Stein's method for the Beta distribution has been developed independently in the paper \cite{GolRei13} as well as in the preprint \cite{Doe12a}.

Although in both works \cite{GolRei13} and \cite{Doe12a} a rate of convergence for the relative number of drawn red balls 
in a P\'{o}lya urn model was derived using Stein's method for the Beta distribution, the actual approaches were quite different. In \cite{GolRei13} the authors developed a useful and widely applicable technique to find a whole class of characterizing operators for a discrete distribution, whose probability mass function is known explicitly, and compared one of these operators to the Stein operator of the limiting Beta distribution. In contrast, the preprint \cite{Doe12a} built on a coupling approach by developing a new version of the \textit{exchangeable pairs approach} of Stein's method for a rather large class of absolutely continuous distributions on the real line. 
This new version of the exchangeable pairs approach differs from that in the framework of the \textit{density method} as developed in \cite{EiLo10} and \cite{ChSh}, since it allows for a modification of the Stein equation, which is adapted to a given exchangeable pair and does not necessarily rely on the characterization by the density method. Recently, in \cite{LRS14}, a nice generalization of the density method, which does not necessarily assume absolute continuity of the given distribution, was given and, 
as an application, it was shown, how the situation of the P\'{o}lya urn example from \cite{GolRei13} may be fitted into this framework.

The main purpose of the present paper is to give a more easily readable account of the method and ideas from \cite{Doe12a} by keeping the class of Beta distributions on $[0,1]$
and the P\'{o}lya urn model as a running example. In addition, we derive new numerical bounds on the solution to the Stein equation for the Beta distribution and, for smooth test functions, also on its first order derivative. For Lipschitz-continuous test functions, these bounds complement those given in \cite{GolRei13} in the sense that they are neither uniformly worse nor uniformly better in the parameters 
of the Beta distribution. Furthermore, we use a new iterative procedure to obtain uniform bounds for derivatives of any order of the solution to the Beta Stein equation with sufficiently smooth right hand side. Incidentally, this is the first paper to give bounds on higher order derivatives of the solution to the Beta Stein equation. It should be mentioned that, generally, obtaining bounds on higher order derivatives of the solution to the Stein equation is 
quite a difficult problem, because the explicit representations of those derivatives become more and more complicated. Hence, to date bounds on higher order derivatives of the solution are still quite rare in Stein's method. For instance, the paper \cite{Daly08} obtains sharp bounds on higher order derivatives in the context of the normal and exponential distributions by exploiting very peculiar identities and facts about these distributions, which are not available for more general absolutely continuous distributions. Also, if one succeeds in deriving a tractable generator representation of the solution to the Stein equation as suggested in \cite{Bar88}, one can usually use this form of the solution to obtain 
bounds on higher order derivatives. This has been used for the multivariate normal \cite{GolRin96} and for the Gamma distribution \cite{Luk}. However, in contrast to the bounds from \cite{Daly08}, these bounds 
usually do not exhibit the smoothness property of the inverse of the corresponding Stein operator. In the case of the multivariate normal distribution with non-singular covariance matrix, one can combine the generator representation with a partial integration to obtain bounds on higher order derivatives, which demand one fewer order of smoothness from the test function than the bounds from \cite{GolRin96}. This has been accomplished independently in \cite{Gau15} and \cite{Doe12c}.   The recent paper \cite{GPR15} combines bounds obtained from the generator representation with the iterative method from the present article 
in order to obtain new bounds on derivatives of arbitrary order of the solution to the Gamma Stein equation, whose dependence on the shape parameter of the Gamma distribution is superior to previous bounds. 

We also indicate, how our iterative method can be applied to obtain bounds for the solution to a Stein equation for the exponential distribution, which are better than those previously obtained. We thus suggest that exploiting this iterative procedure can become a fruitful technique for a larger class of distributions.

The remainder of this paper is structured as follows: In Section \ref{motivation} the general approach is motivated by means of a natural exchangeable pair in the context of the P\'{o}lya urn model and it is stressed 
by means of this example that the framework of exchangeable pairs within the density approach as developed in \cite{EiLo10} and \cite{ChSh} is not always suitable and why one might want to use a different Stein characterization. Furthermore, our main application, Theorem \ref{mt}, a quantitative distributional limit theorem for the relative number of drawn red balls is stated. Then, motivated by this example, in Section \ref{abstract} a general version of Stein's method for a large class of absolutely continuous distributions adapted to a given exchangeable pair is developed. In Section \ref{Beta} the theory from Section \ref{abstract} is specialized to the class of Beta distributions
and Theorem \ref{mt} is proved. Finally, in Section \ref{proofs} several proofs for statements from Sections \ref{abstract} and \ref{Beta} are given. 

\section*{Acknowledgements}
Most parts of the research which led to this article have been accomplished during the authors PhD studies and, hence, there is a certain overlap with the author's PhD thesis \cite{Doe12c}, see also the unpublished paper \cite{Doe12a}. The author was supported by the DFG via SFB/TR 12 during this time. We also refer to Appendix A of \cite{Doe12c} for a version of de l'H\^{o}pital's rule which covers (locally) absolutely continuous functions and which is general enough to justify all invocations of this famous tool within this article. Finally, Appendix B of \cite{Doe12c} contains some identities about the Gibbs sampling procedure, which are generally useful in the exchangeable pairs version of Stein's method and which will be used in the present paper. I am grateful to an anonymous referee whose detailed and valuable comments and suggestions helped me improve the presentation of my results.

\section{The P\'{o}lya urn model and motivation of our general approach}\label{motivation}
The classical P\'{o}lya urn model can be described as follows. Fix positive integers $r,w$ and $c$. At the beginning an urn contains $r$ red balls and $w$ white balls. At each discrete time point $n\in\N$ a ball is drawn from the urn uniformly at random and this ball together with $c$ other balls of the same colour is 
returned to the urn. If we denote by $S_n$ the number of drawn red balls after the first $n$ drawings, $n\in\N$, then we can write 
\begin{equation}\label{repSn}
 S_n=\sum_{j=1}^n X_j\,,
\end{equation}
where $X_j$ denotes the indicator of the event that the $j$-th drawn ball is red, $j\in\N$. It is known from elementary probability theory that for each $n\in\N$ and all $x_1,\ldots,x_n\in\{0,1\}$ we have  
\begin{equation}\label{jdx}
P(X_1=x_1,\ldots,X_n=x_n)=\frac{\prod_{i=0}^{k-1}(r+ci)\prod_{j=0}^{n-k-1}(w+cj)}{\prod_{l=0}^{n-1}(s+w+cl)}\,,
\end{equation}
where $k:=\sum_{j=1}^n x_j$. In particular, this shows that the sequence $(X_j)_{j\in\N}$ is exchangeable.
It now follows from \eqref{jdx} that for each $k=0,\ldots,n$ we have 
\begin{equation*}
P(S_n=k)=\binom{n}{k}\frac{\prod_{i=0}^{k-1}(r+ci)\prod_{j=0}^{n-k-1}(w+cj)}{\prod_{l=0}^{n-1}(s+w+cl)}\,, 
\end{equation*}
or, with $a:=\frac{r}{c}$ and $b:=\frac{w}{c}$,
\begin{equation}\label{pmSn}
P(S_n=k)=\frac{\binom{-a}{k}\binom{-b}{n-k}}{\binom{-a-b}{n}}\,,
\end{equation}
where, for a real number $x$ and a nonnegative intger $m$, we define the generalized binomial coefficient by
\begin{equation*}
 \binom{x}{m}:=\frac{x(x-1)\cdot\dotsc\cdot(x-m+1)}{m!}\,.
\end{equation*}
The distribution of $S_n$ given by \eqref{pmSn} is usually referred to as the \textit{P\'{o}lya distribution} with parameters $n\in\N$ and  $a,b>0$. It is a well-known fact that the distribution of $\frac{1}{n}S_n$ converges weakly as $n\to\infty$ to the distribution $Beta(a,b)$ with parameters $a$ and $b$, where, for general $a,b>0$, the 
Beta distribution $Beta(a,b)$ with parameters $a$ and $b$ is defined by the density function $p:=p_{a,b}$ with 
\begin{equation}\label{densbeta}
p_{a,b}(x):=\begin{cases}
             \frac{1}{B(a,b)}x^{a-1}(1-x)^{b-1},&0<x<1\\
             0,&\text{else.}
            \end{cases}
 \end{equation}
Here, $B(a,b)$ denotes the Euler Beta function $B(a,b)=\int_0^1 x^{a-1}(1-x)^{b-1}dx$ which is related to the Gamma function $\Gamma(t)=\int_0^\infty x^{t-1}e^{-x}dx$ via 
\begin{equation}\label{betagamma}
 B(a,b)=\frac{\Gamma(a)\Gamma(b)}{\Gamma(a+b)}\,.
\end{equation}
From now on denote by 
\begin{equation}\label{W}
 W:=W_n:=\frac{1}{n}S_n=\frac{1}{n}\sum_{j=1}^n X_j
\end{equation}
the relative number of drawn red balls after the first $n$ drawings from the urn. 
Denote by $C^{1,1}([0,1];\R)$ the space of all continuously differentiable real-valued functions on $[0,1]$ which have a Lipschitz-continuous derivative.
\begin{theorem}\label{mt}
Let $Z\sim Beta(a,b)$. For each $h\in C^{1,1}([0,1];\R)$ we have that 
\begin{align*}
 &\Bigl|E[h(W)]-E[h(Z)]\Bigr|\notag\\
 &\leq\frac{C(a,b)}{n}\fnorm{h'}\Biggl(\frac{ab}{a+b}+\frac{(a+b)C(a+1,b+1)}{6}\Bigl(1+\frac{a+b-1}{n}\Bigr)\Biggr)\notag\\
 &\;+\frac{C(a+1,b+1)}{6n}\fnorm{h''}\Bigl(1+\frac{a+b-1}{n}\Bigr)\,,
\end{align*}
where the constants $C(\cdot,\cdot)$ are defined in \eqref{lipcon1} and \eqref{lipcon2} below and $\fnorm{h''}$ denotes the minimum Lipschitz constant of $h'$.
\end{theorem}

The proof will be given in Section \ref{Beta}.
In the paper \cite{GolRei13} the authors even proved a concrete upper bound of order $n^{-1}$ for the Wasserstein distance between the distributions of $Z$ and $W$ from Theorem \ref{mt} and also showed that the rate $n^{-1}$ is optimal. 
Since the Wasserstein distance is induced by $1$-Lipschitz test functions, this implies that their result is stronger than Theorem \ref{mt} as far as the class of test functions is concerned. 
However, it should be mentioned that their method of comparing Stein operators can only be applied in situations, where the distribution of $W$ is explicitly known. Contrarily, the exchangeable pairs technique which is used here, in general, seems to be more flexible in this respect. 
For instance, our plug-in result, Theorem \ref{betaplugin} below, might be beneficial for other applications, where the exact distribution of $W$ is not at hand. Moreover, even in the situation of Theorem \ref{mt} there exist 
parameters $a,b>0$ and test functions $h$ such that our bound is smaller than the one obtained in \cite{GolRei13}. To see this, fix $n$ and let $a=b$ tend to zero. Also, let $h\in C^{1,1}([0,1];\R)$ be such that
$\fnorm{h'}=\fnorm{h''}=1$. Then, as $C(a,a)=4$ if $a\leq 1$ and by continuity of $C(a,a)$ in $a$, we see that the bound given in Theorem \ref{mt} converges to  $\frac{2}{3n}(1-1/n)\leq\frac{2}{3n}$, whereas the bound from \cite{GolRei13} converges to the bigger value $\frac{9}{2n}$.\\

Recall that a pair $(X,X')$ of random elements on a common probability space is called exchangeable, if 
\begin{equation*}
 (X,X')\stackrel{\D}{=}(X',X)\,.
\end{equation*}
Representation \eqref{W} for $W$ suggests constructing another random variable $W'$ such that $W$ and $W'$ make up an exchangeable pair using a Gibbs sampling procedure. Noticing that also the random variables $X_1,\dotsc,X_n$ are exchangeable, the construction of $W'$ can be simplified to the following:\\
Observe $X_1=x_1,\dotsc, X_n=x_n$ and construct $X_n'$ according to the distribution $\calL(X_n|X_1=x_1,\dotsc, X_{n-1}=x_{n-1})$. Then, letting
\begin{equation}\label{W'}
 W':=W-\frac{1}{n}X_n+\frac{1}{n}X_n'
\end{equation}
the pair $(W,W')$ is exchangeable. Note that $\abs{W-W'}\leq\frac{1}{n}$ is small which suggests that the exchangeable pair $(W,W')$ be beneficial for a Stein's method approach to the proof of weak convergence of $\calL(W_n)$ to $Beta(a,b)$. From the exchangeable pairs approach within normal approximation 
(see e.g. \cite{St86}, \cite{CheShaSing} or \cite{CGS}) and for non-normal approximation (see \cite{EiLo10} and \cite{ChSh}) we know that exchangeability of $(W,W')$ is not enough to guarantee distributional closeness of $W$ and of $Z\sim Beta(a,b)$ but that a further \textit{regression property} has to be satisfied.  

\begin{prop}\label{polyaprop1}
The exchangeable pair $(W,W')$ satisfies the regression property
\[E\bigl[W'-W|W\bigr]=\frac{a+b}{n(a+b+n-1)}\Bigl(\frac{a}{a+b}-W\Bigr)=\la\ga_{a,b}(W)\,,\]
where $\ga_{a,b}(x)=(a+b)\bigl(\frac{a}{a+b}-x\bigr)$ and $\la=\la_n= \frac{1}{n(a+b+n-1)}$.
\end{prop}
\begin{proof}
We have $W'-W=\frac{X_n'}{n}- \frac{X_n}{n}$ and by exchangeability of $X_1,\ldots,X_n$ it clearly holds that 
$E[X_n|W]=E[X_n|S_n]=\frac{1}{n}S_n=W$. Also, by the definition of $X_n'$ and since $X_n'$ only assumes the values $0$ and $1$ we have 
for any $x_1,\ldots,x_{n-1}\in\{0,1\}$
\begin{align*}
&E[X_n'|X_1=x_1,\ldots,X_n=x_n]=E[X_n|X_1=x_1,\ldots,X_{n-1}=x_{n-1}]\\
&=P(X_n=1|X_1=x_1,\ldots,X_{n-1}=x_{n-1})=\frac{r+c\sum_{j=1}^{n-1}x_j}{r+w+c(n-1)}\,,
\end{align*}
and hence,
\begin{align*}
 E[X_n'|X_1,\ldots,X_n]=\frac{r+c\sum_{j=1}^{n-1}X_j}{r+w+c(n-1)}=\frac{r+cnW-cX_n}{r+w+c(n-1)}\,.
\end{align*}
Thus, since $\sigma(W)\subseteq\sigma(X_1,\ldots,X_n)$, we obtain
\begin{align*}
E[X_n'|W]&=E\Bigl[E\bigl[X_n'|X_1,\ldots,X_n\bigr]\,|W\Bigr]=\frac{r+cnW-cW}{r+w+c(n-1)}\\
&=\frac{r+c(n-1)W}{r+w+c(n-1)}=\frac{a+(n-1)W}{a+b+n-1}\,.
\end{align*}
Finally, we have
\begin{align*}
E[W'-W|W]&=\frac{1}{n}E[X_n'-X_n|W]=\frac{1}{n}\frac{a+(n-1)W}{a+b+n-1}-\frac{1}{n}W\\
&= \frac{a-(a+b)W}{n(a+b+n-1)}=\frac{a+b}{n(a+b+n-1)}\Bigl(\frac{a}{a+b}-W\Bigr)\,,
\end{align*}
as was to be shown.\\
\end{proof}

From the theory developed in \cite{EiLo10} and in \cite{ChSh} we know that if a given exchangeable pair $(W,W')$ satisfies a regression property of the form 
\begin{equation}\label{densreg}
 \frac{1}{\lambda}E\bigl[W'-W\bigl|W\bigr]=\psi(W)+R\,,
\end{equation}
where $\lambda>0$ is a typically small constant and $R$ is negligible in size, then $\calL(W)$ can be approximated by the absolutely continuous distribution whose density has logarithmic derivative $\psi$, if and only if the following additional condition is satisfied: It must be the case that 
\begin{equation}\label{lln}
 \frac{1}{2\lambda}E\bigl[(W'-W)^2\bigl|W\bigl]\approx 1\,,
\end{equation}
which is often paraphrased as that the term on the left hand side in \eqref{lln} must satisfy a law of large numbers in order for the approximation to be accurate. Comparing \eqref{densreg} to the statement of Proposition \ref{polyaprop1} we see that according to the theory from \cite{EiLo10} or \cite{ChSh} the only possibility 
would be to approximate the distribution of $W$ by a distribution whose density has logarithmic derivative equal to (a constant multiple) of 
\begin{equation*}
 \frac{a}{a+b}-x\,,
\end{equation*}
for $x$ in the support of this density, which should be equal to $[0,1]$ in this case.
Since the logarithmic derivative $\psi_{a,b}$ of the density $p_{a,b}$ of $Beta(a,b)$ is given by 
\begin{equation}\label{logdevbeta}
 \psi_{a,b}(x)=\frac{d}{dx}\log p_{a,b}(x)=\frac{p_{a,b}'(x)}{p_{a,b}(x)}=\frac{a-1-(a+b-2)x}{x(1-x)}\,,\quad 0<x<1\,,
\end{equation}
and we already know that 
\[W_n\stackrel{\D}{\longrightarrow}Beta(a,b)\quad\text{as }n\to\infty\,,\]
we conclude by way of contradiction that the law of large numbers \eqref{lln} cannot hold. Indeed, we will see in Proposition \ref{polyaprop2} below that that the term on the left hand side of \eqref{lln} is close to the non-constant random quantity $W(1-W)$ rather than to the constant $1$.
From Proposition \ref{polyaprop1} and some experience with the exchangeable pairs approach within Stein's method we conclude that it would be desirable to have a Stein operator $L$ of the form 
\begin{equation}\label{steinopbeta}
Lg(x)=\eta_{a,b}(x)g'(x)+\gamma_{a,b}(x)g(x) 
\end{equation}
for the Beta distribution $Beta(a,b)$.
Indeed, in Section \ref{Beta} we will see that a random variable $Z\sim Beta(a,b)$ satisfies the Stein identity 
\begin{equation}\label{steinidbeta}
 E\Bigl[Z(1-Z)g'(Z)+(a+b)\Bigl(\frac{a}{a+b}-Z\Bigr)g(Z)\Bigr]=0
\end{equation}
for all $g$ in a suitable class of functions, i.e. we can let $\eta_{a,b}(x)=\eta(x)=x(1-x)$. Evidently, the Stein identity \eqref{steinidbeta} was first found in \cite{Sch01} and it was also used in \cite{GolRei13}.
The statement of the following Proposition will make it possible to exploit the above constructed exchangeable pair $(W,W')$ in connection with the Stein identity \eqref{steinidbeta} in Section \ref{Beta}.
\begin{prop}\label{polyaprop2}
For the above constructed exchangeable pair $(W,W')$ we have
\[E\bigl[(W'-W)^2|W\bigr]= \frac{(2n+b-a)W-2nW^2+a }{n^2(a+b+n-1)}\]
and hence
\[\frac{1}{2\la}E\bigl[(W'-W)^2|W\bigr]=W(1-W)+\frac{b-a}{2n}W+\frac{a}{2n}\,.\]
\end{prop}

\begin{proof}
From general facts about Gibbs sampling (see e.g. Appendix B in \cite{Doe12c}) it is known that
\begin{align*}\label{polyaeq1}
 E\bigl[(W'-W)^2|W\bigr]&=\frac{1}{n^2}\Bigl(E[X_n|W]+E\bigl[E[X_n^2|X_1,\ldots,X_{n-1}]\,|\,W\bigr]\\
&\,-2E\bigl[X_nE[X_n|X_1,\ldots,X_{n-1}]|W\bigr]\Bigr)\,.\nonumber
\end{align*}
Since $X_n^2=X_n$ we have from the proof of Proposition \ref{polyaprop1} that
\[E[X_n^2|X_1,\ldots,X_{n-1}]=E[X_n|X_1,\ldots,X_{n-1}]=\frac{a+nW-X_n}{a+b+n-1}\,,\]
and hence
\[E\bigl[E[X_n^2|X_1,\ldots,X_{n-1}]\,|\,W\bigr]=\frac{a+(n-1)W}{a+b+n-1}\,,\]
where we have used $E[X_n|W]=W$ again. Finally, we compute 
\begin{align*}
E\bigl[X_nE[X_n|X_1,\ldots,X_{n-1}]|W\bigr]&=\frac{1}{a+b+n-1}E\bigl[aX_n+nWX_n-X_n^2\,\bigl|\,W\bigr]\\
&=\frac{aW+nW^2-W}{a+b+n-1}=\frac{(a-1)W+nW^2}{a+b+n-1}\,.
\end{align*}
Putting pieces together, we eventually obtain
\begin{align}\label{pprop2eq1}
 E\bigl[(W'-W)^2|W\bigr]&=\frac{1}{n^2}\Bigl(W+\frac{a+(n-1)W}{a+b+n-1}-2\frac{(a-1)W+nW^2}{a+b+n-1}\Bigr)\notag\\
&=\frac{(2n+b-a)W-2nW^2+a }{n^2(a+b+n-1)}\,.
\end{align}
The last assertion easily follows from \eqref{pprop2eq1} and from $\la= \frac{1}{n(a+b+n-1)}$.\\
\end{proof}

One main aspect of the theoretical contribution of this article is to emphasize that it is no coincidence that 
\begin{equation*}
\frac{1}{2\lambda}E\bigl[(W'-W)^2\bigl|W\bigl]\approx \eta(W)=W(1-W)\,,
\end{equation*}
but that this is a natural replacement of condition \eqref{lln} from the density approach to our class of Stein operators 
of the form \eqref{gensteinop} below.\\
We end this motivational section by an abstraction of the ideas in the context of the P\'{o}lya urn model and the limiting Beta distribution above. Suppose we are given a sequence of random variables $W=W_n$ of which we know that, as $n\to\infty$, it converges in distribution to a random variable $Z$ with an absolutely continuous distribution 
and density $p$ with respect to the Lebesgue measure. We will also assume that $p$ itself is absolutely continuous (on each compact subinterval of its support $\abquer$, where $-\infty\leq a<b\leq\infty$ are extended real numbers). Suppose also that we can naturally
construct a random variable $W'$, a small random perturbation of $W$, such that $(W,W')$ is an exchangeable pair, $\abs{W-W'}$ is small in a certain sense and that a regression property of the form 
\begin{equation}\label{genreg}
\frac{1}{\lambda}E\bigl[W'-W\,\bigl|\,W\bigr]=\gamma(W)+R 
\end{equation}
holds, where $\gamma$ is a certain function on the support of $\calL(Z)$, $\la>0$ is constant and $R$ is a negligible remainder term. The goal is to compute a rate of convergence for the distributional convergence $W\rightarrow Z$ by Stein's method of exchangeable pairs for $\calL(Z)$. 
By the above reasoning it would be beneficial to have a characterizing Stein operator $L$ for $Z$ of the form
\begin{equation}\label{gensteinop}
 Lg(x)=\eta(x)g'(x)+\gamma(x)g(x)\,,
\end{equation}
where $\eta$ is a function that still has to be found. One might suppose that, in order that $L$ characterizes $\calL(Z)$, given the density $p$ of $Z$ and the function $\gamma$ the function $\eta$ is unique but we will see that this is only so up to a constant multiple of $p^{-1}$. 
Note that by exchangeability 
\begin{equation}\label{Egamma}
0=\frac{1}{\lambda}E[W'-W]=E[\gamma(W)]+E[R]\approx E[\gamma(W)]\approx E[\gamma(Z)]\,,
\end{equation}
where the first approximation is by the assumption that $R$ is of negligible order and the second is by the fact that $W$ converges to $Z$ in distribution. Hence, it is natural to assume from the outset that $E[\gamma(Z)]=0$. In particular, we should assume that $E\abs{\gamma(Z)}<\infty$. A natural question is, 
given $p$ and $\gamma$, if there is a general formula for the function $\eta$. In the preprint \cite{Doe12a} the first order linear differential equation 
\begin{equation}\label{etaode}
 \eta'(x)=\gamma(x)-\psi(x)\eta(x)
\end{equation}
was found by making, for a given test function $h$, the ansatz $g_h(x)=\alpha(x)f_h(x)$ for the solutions $g_h$ of the Stein equation 
\begin{equation}\label{gensteineq}
 \eta(x)g'(x)+\gamma(x)g(x)=h(x)-E[h(Z)]
\end{equation}
belonging to the operator \eqref{gensteinop} and $f_h$ of the Stein equation
\begin{equation}\label{denssteineq}
 f'(x)+\psi(x)f(x)=h(x)-E[h(Z)]
\end{equation}
corresponding to the density approach. Here, again $\psi$ denotes the logarithmic derivative of $p$. In this paper we follow a different, more direct reasoning. If $\eta$ is such that \eqref{gensteinop} is characterizing $\calL(Z)$, then, for suitable functions $g$ by partial integration:
\begin{align}\label{etaheu}
 E\bigl[\eta(Z)g'(Z)\bigr]&=\int_a^b\eta(x)p(x)g'(x)dx=g\eta p\bigl|_a^b-\int_a^b\bigl(\eta'(x)p(x)+p'(x)\eta(x)\bigr)g(x)dx\notag\\
 &=g\eta p\bigl|_a^b-E\bigl[\bigl(\eta'(Z)+\psi(Z)\eta(Z)\bigr)g(Z)\bigr]\,.
 \end{align}
Thus, if we want this expression to equal 
\begin{equation*}
 g\eta p\bigl|_a^b-E\bigl[\gamma(Z)g(Z)\bigr]\,,
\end{equation*}
then from \eqref{etaheu} we conclude that $\eta$ must satisfy \eqref{etaode}. Of course, \eqref{etaode} can be solved by the method of variation of the constant and it turns out that 
\begin{equation}\label{defeta}
 \eta(x):=\frac{1}{p(x)}\int_a^x\gamma(t)p(t)dt\,,\quad a<x<b\,,
\end{equation}
is a particular solution which even satisfies $(\eta p)(a+)=(\eta p)(b-)=0$ whenever $E[\gamma(Z)]=0$ and, hence, the boundary conditions 
\begin{equation}\label{boundary}
(g\eta p)(a+)=\lim_{x\downarrow a}g(x)\eta(x)p(x)=0=\lim_{x\uparrow b}g(x)\eta(x)p(x)=(g\eta p)(b-)\,, 
\end{equation}
hold for each regular enough, say e.g. bounded, function $g$.  
Also note that every other solution to \eqref{etaode} has the form 
\begin{equation*}
\eta_\kappa(x)=\eta(x)+\frac{\kappa}{p(x)} 
\end{equation*}
for some constant $\kappa$. In principle, the particular choice of $\kappa$ is arbitrary and the choice $\kappa\not=0$ sometimes even yields better behaved solutions $g_h$ to the Stein equation \eqref{gensteineq}. In fact, it is easy to see from \eqref{gensteinsol} below that the choice $\kappa\not=0$ automatically implies $g_h(a+)=g_h(b-)=0$.
 Also, sometimes the choice $\kappa\not=0$ is implicit in the density approach. For instance, if $a>-\infty$, $b<\infty$ and the density $p$ is such that $0\not=p(a+)=p(b-)\in\R$, then one can easily see that $\ga(x):=\psi(x)$ satisfies 
 \begin{equation*}
  E[\ga(Z)]=\int_a^b p'(x)dx=p(b-)-p(a+)=0
 \end{equation*}
 but 
 \begin{equation*}
 \eta(x)=\frac{p(x)-p(a+)}{p(x)}\not=1\,,\quad a<x<b\,. 
 \end{equation*}
Hence, in all these cases, using the density approach implicitly entails choosing $\eta_\kappa$ with $\kappa=p(a+)$. When developing the general theory in Section \ref{abstract} we restrict ourselves to the solution $\eta$ given by \eqref{defeta}, i.e to $\kappa=0$. 
We thus already mention at this point that the density approach for $p$ is included in the theory presented in Section \ref{abstract} if and only if 
\begin{equation*}
p(a+)=p(b-)=0\,.
\end{equation*}
However, at least if $\ga(x)=c(E[Z]-x)$, it turns out that in many cases $\eta$ given by \eqref{defeta} has a neat analytical representation, e.g. it is given by a polynomial of degree 
at most $2$, whereas the choice $\kappa\not=0$ would introduce a complicated coefficient into \eqref{gensteineq} originating from the term $p(x)^{-1}$. For instance, if $Z\sim N(0,1)$ is standard normally distributed and $\gamma(x)=-x$, then \eqref{defeta} yields $\eta\equiv1$, whereas the general expression 
is $\eta_\kappa(x)=1+\kappa e^{x^2/2}$, which is difficult to handle in practice. Furthermore, if $p$ is not bounded away from zero, then $\kappa\not=0$ gives an unbounded function $\eta_\kappa$, whereas $\eta$ given by \eqref{defeta} usually is bounded, at least if $a>-\infty$ and $b<\infty$ (see, e.g. Proposition \ref{etaprop} below).\\
In the next section we will see that under certain mild conditions on the density $p$ of $Z$ and on the coefficient $\gamma$ which, of course, needs not originate from an exchangeable pair, the operator 
$L$ given by \eqref{gensteinop} is indeed characterizing $\calL(Z)$ and prove bounds on the corresponding Stein equation \eqref{gensteineq} for suitable test functions $h$.
Finally, we want to propose a strategy of how to proceed, if, contrarily to the above reasoning, we do not know the limiting density $p$ from the outset but are only given 
an exchangeable pair $(W,W')$ such that \eqref{genreg} holds and also 
\begin{equation}\label{secmom}
 \frac{1}{2\lambda}E\bigl[(W'-W)^2\bigl|W\bigr]=\eta(W)+S
\end{equation}
is satisfied with the same constant $\lambda>0$ and a small remainder term $S$, where $\eta$ now is a certain given function, which is positive on a certain open interval $J=(a,b)\subseteq\R$, where $\gamma$ is also defined. Note that from \eqref{etaode} we have for the logarithmic derivative $\psi$ of the sought density $p$ that 
\[\psi=\frac{\gamma-\eta'}{\eta}\]
and, hence, for $x\in J$ and $x_0\in J$ an arbitrary point, we have
\begin{align}\label{densform1}
 p(x)&=p(x_0)\exp\Bigl(\int_{x_0}^x\psi(t)dt\Bigr)=\frac{p(x_0)\eta(x_0)}{\eta(x)}\exp\Bigl(\int_{x_0}^x\frac{\gamma(t)}{\eta(t)}dt\Bigr)\notag\\
 &=\frac{K}{\eta(x)}\exp\Bigl(\int_{x_0}^x\frac{\gamma(t)}{\eta(t)}dt\Bigr)\,.
 \end{align}

 Here, of course, $K=p(x_0)\eta(x_0)$ is the normalization constant. Formula \eqref{densform1} shows that $p$ is uniquely determined by $\gamma$ and $\eta$. Furthermore, in Theorem \ref{etagammatheo} we will give precise criteria for $\gamma$ and $p$ defined by \eqref{densform1} to satisfy 
 \[\int_a^b\gamma(t)p(t)dt=0\]
and for $\eta$ to satisfy \eqref{defeta} so that the results of the theory developed in Section \ref{abstract} can in fact be applied. This, together with Proposition \ref{genpluginprop1} and Remark \ref{plugrem} (iii), suggests the approximation of $\calL(W)$ by the distribution with density $p$, if the exchangeable pair
$(W,W')$ satisfies \eqref{genreg} and \eqref{secmom}. Note that this idea yields a certain extension of the methodology proposed in \cite{ChSh}, where only Stein characterizations from the density approach are put to use.    

\section{The general approach}\label{abstract}
Motivated by Section \ref{motivation} in this section we develop a general version of Stein's method for a random variable $Z$ with an absolutely continuous distribution with respect to the Lebesgue measure. This version is useful for those distributions, which allow for a tractable first order linear Stein operator. This class covers many of the standard absolutely continuous distributions. However, it should not be left unmentioned that certain distributions, like the Laplace \cite{PiRen12}, the Variance-Gamma \cite{Gau14} and the PRR distribution \cite{PRR13} fall outside the scope of this approach, as they only possess a second order linear Stein operator with tractable coefficients.

For an interval $J\subseteq\R$, we will call a function defined on $\R$ \textit{locally absolutely continuous on $J$}, if its restriction to each compact sub-interval of $J$ is absolutely continuous. Also, we will use the words increasing, decreasing and so on in the weak sense, unless explicitly otherwise stated. 
Througout we suppose that $Z$ has a Lebesgue density $p$ on $\R$ satisfying the following condition: 
\begin{cond}\label{gencond1}
For some extended real numbers $-\infty\leq a<b\leq\infty$
the density $p$ is positive and locally absolutely continuous on the interval $(a,b)$.
\end{cond}
By $\abquer$ we will henceforth denote the closure of the real interval $(a,b)$ with respect to the usual topology on $\R$.
Furthermore, we assume that we are given a function $\gamma$ on $\abquer$ which might be motivated by a given exchangeable pair and which has the following properties:

\begin{cond}\label{gencond2}
The function $\ga:\abquer\rightarrow\R$ is such that 
\begin{enumerate}[(i)]
\item $\ga$ is Borel-measurable and not identically equal to zero,
\item $\ga$ is decreasing on $\abquer$,
\item $E\abs{\ga(Z)}=\int_a^b |\ga(t)|p(t)dt<\infty$ and in fact $E[\ga(Z)]=\int_a^b\ga(t)p(t)dt=0$.
\end{enumerate} 
\end{cond}

Henceforth, we will always assume that Conditions \ref{gencond1} and \ref{gencond2} are satisfied. Note that by Condition \ref{gencond2} there exists a point $x_0\in(a,b)$ such that 
\begin{equation}\label{gapn}
\ga(x)\geq0\quad\text{if}\quad a<x<x_0\quad\text{and}\quad\ga(x)\leq0\quad\text{if}\quad x_0<x<b\,,
\end{equation}
though it might not be unique.
For definiteness, we choose
\begin{equation}\label{defx0}
 x_0:=\sup\{x\in(a,b)\,:\,\ga(x)>0\}\,.
\end{equation}

By item (iii) of Condition \ref{gencond2} we can define the function $I:\abquer\rightarrow\R$ by
\begin{equation}\label{defI}
 I(x):=\int_a^x \gamma(t)p(t)dt=-\int_x^b\gamma(t)p(t)dt
\end{equation}
and by the positivity of $p$ on $(a,b)$ we can define the function $\eta$ on $(a,b)$ by 
\begin{equation}\label{etadef}
 \eta(x):=\frac{I(x)}{p(x)}=\frac{1}{p(x)}\int_a^x \gamma(t)p(t)dt=-\frac{1}{p(x)}\int_x^b \gamma(t)p(t)dt\,.
\end{equation}
The following proposition lists some properties of the function $I$.

\begin{prop}\label{genprop1}
Under Conditions \ref{gencond1} and \ref{gencond2} the function $I$ has the following properties:
\begin{enumerate}[{\normalfont (a)}]
 \item $I$ is locally absolutely continuous on $\abquer$.
 \item $I$ is nonnegative and $I(a+)=I(b-)=0$.
 \item $I$ is increasing on $\overline{(a,x_0]}$ and decreasing on $\overline{[x_0,b)}$ and, hence, attains its global maximum at $x_0$. 
\end{enumerate}
\end{prop}
\begin{proof}
 Of course, (a) follows from the fundamental theorem of calculus for Lebesgue integration and the second part of (b) is immediate from item (iii) of Condition \ref{gencond2}. Finally, (c) and the first part of (b) follow from the second part of (b) and \eqref{gapn}.\\
\end{proof}

If $a>-\infty$ and/or $b<\infty$, then it is of interest to know under what circumstances it is possible to extend $\eta$ to a continuous function on $\abquer$ because we would like to have $\eta(W)$ make sense, even if $W$ assumes one of the boundary values $a$ and $b$ with positive probability. 
We will see that in most cases we indeed have $\eta(a+)=0$ or $\eta(b-)=0$ if $a>-\infty$ or if $b<\infty$, respectively. The following Mills ratio condition is satisfied by most absolutely continuous distributions and will in fact turn out to be equivalent to the asserted boundary behaviour of $\eta$. From now on, we will denote by $F$
the distribution function corresponding to the density $p$.

\begin{cond}\label{gencond3}
The density $p$ of $Z$ satisfies all the properties from Condition \ref{gencond1} and also the following:
\begin{enumerate}[(i)]
\item If $a>-\infty$, then $\lim_{x\downarrow a}\frac{F(x)}{p(x)}=0$.
\item If $b<\infty$, then $\lim_{x\uparrow b}\frac{1-F(x)}{p(x)}=0$.
\end{enumerate}
\end{cond}

\begin{prop}\label{etaprop}
Assume that Conditions \ref{gencond1} and \ref{gencond2} hold for $p$ and $\ga$, respectively. Then, the function $\eta$ vanishes at the finite end points of the support $\abquer$ of $\calL(Z)$, i.e. $\eta(a+)=0$ whenever $a>-\infty$ and $\eta(b-)=0$ whenever $b<\infty$,  if and only if Condition \ref{gencond3} is satisfied.
Thus, in this case we can extend $\eta$ to a continuous function on $\abquer$ vanishing at the finite end points of this interval. 
\end{prop}

Not every density $p$ satisfies Condition \ref{gencond3} as is clarified by the following example.

\begin{example}\label{ce}
Let $\delta_n\in(0,1)$, $n\geq1$, be such that $\sum_{n\geq1}\delta_n=1$ and define $x_n:=1-\sum_{j=1}^{n-1}\delta_j=\sum_{j=n}^\infty \delta_j$ and $I_n:=[x_{n+1},x_n]$, $n\geq1$. Furthermore let $q$ be the unique continuous function, which is linear on each interval $I_n$ and 
such that $q(x_{2n})=\delta_{2n}^2$ and $q(x_{2n+1})=\delta_{2n}$ for $n\geq1$ and $q(1):=\delta_1$. Define $p$ to be the probability density which is a constant multiple of $q$. Then, $p$ satisfies Condition \ref{gencond1} with $a=0$ and $b=1$ but Condition \ref{gencond3} does not hold: 
We have $\lim_{n\to\infty}x_{2n}=0$ but 
\begin{align*}
 \frac{F(x_{2n})}{p(x_{2n})}&\geq\frac{F(x_{2n})-F(x_{2n+1})}{p(x_{2n})}=\frac{1}{p(x_{2n})}\int_{x_{2n+1}}^{x_{2n}}p(t)dt\\
 &=\frac{\delta_{2n}\bigl(p(x_{2n})+p(x_{2n+1})\bigr)}{2p(x_{2n})}\geq\frac{\delta_{2n}p(x_{2n+1})}{2p(x_{2n})}
=\frac{1}{2}\,.
\end{align*}
Note that $p$ satisfies $\lim_{x\to0}p(x)=0$, so that this does not only happen because $p(0+)$ might not exist.
\end{example}

The counterexample given in Example \ref{ce} is quite artificial. Indeed, the following proposition lists mild assumptions on the density $p$ which guarantee that Condition \ref{gencond3} is satisfied. In practice, at least one of these assumptions is usually met. In particular, note that by part (f) of Proposition \ref{propbs} the Mills ratio limits from Condition \ref{gencond3} at finite boundary points $a$ or $b$ are always zero, whenever they exist. 

\begin{prop}\label{propbs}
Assume $a>-\infty$. In either of the following cases  $\lim_{x\downarrow a}\frac{F(x)}{p(x)}=0$.
\begin{enumerate}[{\normalfont (a)}]
 \item The density $p$ is bounded away from zero in a suitable neighbourhood of $a$.
 \item We have $p(a+)=0$ and there is a $\delta>0$ such that $p$ is increasing on $(a,a+\delta)$.
 \item We have $p(a+)=0$ and there is a $\delta>0$ such that $p$ is convex on $(a,a+\delta)$.
 \item We have $p(a+)=0$ and there is a $\delta>0$ such that $p$ is concave on $(a,a+\delta)$.
 \item The density $p$ is analytic at $a$.
 \item The limit  $\lim_{x\downarrow a}\frac{F(x)}{p(x)}$ exists.
 \end{enumerate}
Of course, similar conditions guarantee that $\lim_{x\uparrow b}\frac{1-F(x)}{p(x)}=0$ if $b<\infty$. 
\end{prop}
The proof is given in Section \ref{proofs}.
Now, for a given Borel-measurable function $h$ on $\abquer$ such that $E\abs{h(Z)}<\infty$ consider the Stein equation \eqref{gensteineq}. It is easy to check that the function\\ $g_h:(a,b)\rightarrow\R$ given by 
\begin{align}\label{gensteinsol}
g_h(x)&:=\frac{1}{p(x)\eta(x)}\int_a^x\bigl(h(t)-E[h(Z)]\bigr)p(t)dt\nonumber\\
&=-\frac{1}{p(x)\eta(x)}\int_x^b\bigl(h(t)-E[h(Z)]\bigr)p(t)dt
\end{align}
is a solution to \eqref{gensteineq} in the sense that $g_h$ is locally absolutely continuous on $(a,b)$ and \eqref{gensteineq} holds for each point $x\in(a,b)$ where $g_h$ is in fact differentiable. At all other points $x\in(a,b)$, in contrast to the usual convention, we define $g_h'(x)$ by \eqref{gensteineq} such that 
this identity holds true on $(a,b)$. The formula for $g_h$ might as well be found by the method of variation of the constant using the fact that $\log(p \eta)$ is a primitive function of $\gamma/\eta$, which in turn follows from \eqref{etaode}. In what follows we will always call the solution $g_h$ given by \eqref{gensteinsol} the 
\textit{standard solution} to equation \eqref{gensteineq}. It turns out that this particular solution has the best properties. For instance,  if $g_h$ is bounded then it is the only bounded solution since the solutions of the corresponding homogeneous equation are given by constant multiples of $I^{-1}=\eta^{-1}p^{-1}$, which is unbounded 
by Proposition \ref{genprop1} (b). Since we do not exclude cases where the given random variable $W$ assumes the finite value $a$ and/or $b$ with positive probability, we have to make sure that $g_h$ can be extended to a continuous function on $\abquer$. Assume that $a>-\infty$. Here, and in what follows we will often write $\htilde$ for 
$h-E[h(Z)]$. By de l'H\^{o}pital's rule we have 
\begin{equation}\label{conta1}
\lim_{x\downarrow a}g_h(x)= \lim_{x\downarrow a}\frac{\int_a^x\htilde(t)p(t)dt}{I(x)}
=\lim_{x\downarrow a}\frac{\htilde(x)p(x)}{\ga(x)p(x)}=\lim_{x\downarrow a}\frac{\htilde(x)}{\ga(x)}
=\frac{h(a+)-E[h(Z)]}{\ga(a+)}\,,
\end{equation}
if $h$ has a right limit at $a$. Note that $\gamma$ has a right limit at $a$ since it is decreasing. Similarly, 
\begin{equation*}
 \lim_{x\uparrow b}g_h(x)= \frac{h(b-)-E[h(Z)]}{\ga(b-)}\,,
\end{equation*}
if $h$ has a left limit at $b<\infty$.

\begin{prop}\label{genprop2}
Assume Conditions \ref{gencond1} and \ref{gencond2} and let $h:\abquer\rightarrow\R$ be a Borel-measurable function such that $E\abs{h(Z)}<\infty$. Then, the standard solution $g_h$ to \eqref{gensteineq} given by \eqref{gensteinsol} has the following properties:
\begin{enumerate}[{\normalfont(a)}]
 \item If $a>-\infty$ and $h$ has a right limit at $a$, then $g_h$ can be extended continuously to $a$ by letting $\displaystyle g_h(a):=\frac{h(a+)-E[h(Z)]}{\ga(a+)}$.
 \item If $b<\infty$ and $h$ has a left limit at $b$, then $g_h$ can be extended continuously to $b$ by letting $\displaystyle g_h(b):=\frac{h(b-)-E[h(Z)]}{\ga(b-)}$.
\end{enumerate}
\end{prop}

The success of Stein's method within applications considerably depends on good bounds on the solutions $g_h$ and their lower order derivatives, generally uniformly over some given class of test functions $h$. 
The next step will be to prove such bounds. It has to be mentioned that we cannot expect to derive concrete good bounds in full generality, but that sometimes further conditions have to be imposed either on the density $p$ or on the coefficient $\ga$.
Nevertheless, we will derive bounds involving functional expressions which can be simplified, computed or further bounded a posteriori for concrete distributions. So our abstract viewpoint will pay off. Moreover, some of our general bounds will already be explicit.\\
In what follows, we denote by $g_h$ the standard solution to Stein's equation \eqref{gensteineq} on $\abquer$, implicitly assuming that $h$ satisfies the assumptions of Proposition \ref{genprop2}. Furthermore, for a function $f$ we denote by $\fnorm{f}$ its essential supremum norm on $\abquer$. Note that this implies 
for $f$ a Lipschitz-continuous function on $\abquer$ that $\fnorm{f'}$ is just its minimum Lipschitz constant. First we give bounds for bounded and measurable test functions $h$.

\begin{prop}\label{genprop3}
Assume Conditions \ref{gencond1} and \ref{gencond2} and let $m$ be the median of $\calL(Z)$. Then, if $h:\abquer\rightarrow\R$ is Borel-measurable and bounded we have
\begin{equation}\label{genboundbounded}
 \fnorm{g_h}\leq\frac{\fnorm{h-E[h(Z)]}}{2I(m)}=\frac{\fnorm{h-E[h(Z)]}}{2\int_a^{m}\ga(t)p(t)dt}.
\end{equation}
\end{prop}

The proof is deferred to Section \ref{proofs}.
The following corollary specializes this result to the case that $\ga(x)=-c(x-E[Z])$ and that $\calL(Z)$ is symmetric with respect to its mean $E[Z]$, i.e. $Z-E[Z]\stackrel{\D}{=}E[Z]-Z$. Then, it is also clear that $m=E[Z]$.

\begin{cor}\label{gencor1}
In addition to Conditions \ref{gencond1} and \ref{gencond2} assume that the distribution $\calL(Z)$ is symmetric with respect to $m=E[Z]$ and that 
$\ga(x)=-c(x-E[Z])$ for some positive constant $c$. Then, for each bounded and  Borel-measurable test function $h:\abquer\rightarrow\R$ we have
\begin{equation}\label{boundboundedsym}
\fnorm{g_h}\leq\frac{\fnorm{h-E[h(Z)]}}{cE[\abs{Z-m}]}.
\end{equation} 
\end{cor}
   
\begin{proof}
In this case we clearly have $I(m)=\frac{c}{2} E[\abs{Z-m}]$ which implies the result by Proposition \ref{genprop3}.\\
\end{proof}
In the case that $Z\sim N(0,1)$ and $c=1$ this result specializes to the well known bound $\fnorm{g_h}\leq\sqrt{\frac{\pi}{2}}\fnorm{h-E[h(Z)]}$ (see \cite{CGS} or \cite{CheShaSing}, e.g.).

\begin{remark}\label{genrem2}
\begin{enumerate}[(a)]
 \item In the statement of Proposition \ref{genprop3} it might suprise that there is no bound mentioned for $\fnorm{g_h'}$. This is because, in general, a bound of the form $\fnorm{g_h'}\leq C \fnorm{h-E[h(Z)]}$ with a finite constant $C$ does not exist in this setup. For instance, for $z>0$ and $Z$ having the exponential distribution with mean one, consider the Stein equation
\begin{equation}\label{expse}
 xg'(x)+(1-x)g(x)=1_{[0,z]}(x)-P(Z\leq z)\,.
\end{equation}
Identity (3.3) from \cite{CFR11} shows that for $x>z$ the solution $g_z$ to \eqref{expse} satisfies
\begin{equation*}
 g_z'(x)=\frac{e^{-z}-1}{x^2}\,.
\end{equation*}
Hence, we have that 
\begin{equation*}
 \sup_{x>z}\abs{g_z'(x)}=\frac{1-e^{-z}}{z^2}\stackrel{z\downarrow0}{\longrightarrow}\infty\,,
\end{equation*}
proving that such a constant $C$ in general cannot exist. Note also that this is contrary to the density approach, where one usually
has such a bound (see \cite{ChSh} or \cite{CGS}).  
\item The Kolmogorov distance between a given random variable $W$ and $Z$ is induced by the class of test functions $h_z:=1_{(-\infty,z]}$, where $z\in(a,b)$. In this situation it is easy to verify that the standard solution $g_z:=g_{h_z}$ to \eqref{gensteineq} is given by 
\begin{equation*}
 g_z(x)=\begin{cases}
         \frac{F(x)(1-F(z))}{I(x)}\,,& a<x\leq z\\
         \frac{F(z)(1-F(x))}{I(x)}\,,& z<x<b
        \end{cases}\quad\text{and}\quad
        \fnorm{g_z}=\frac{F(z)(1-F(z))}{I(z)}=:S(z)\,.
\end{equation*}
By using de l'H\^{o}pital's rule it is not hard to check that always $\sup_{z\in(a,b)}S(z)<\infty$. 
Furthermore, $g_z$ is Lipschitz-continuous and on $(a,b)\setminus\{z\}$ it is infinitely often continuously differentiable with
\begin{equation*}
 g_z'(x)=\begin{cases}
          \frac{(1-F(z))p(x)H(x)}{I(x)^2}\,,& a<x< z\\
         \frac{-F(z)p(x)G(x)}{I(x)^2}\,,& z<x<b\,,
         \end{cases}
\end{equation*}
where the functions $H$ and $G$ are defined in Proposition \ref{genprop4}.
From the negative example of (a) we already know that, in general, there is no finite constant $C$ such that 
\begin{equation*}
 \fnorm{g_z'}\leq C\,,\quad a<z<b\,.
\end{equation*}
Nevertheless, even in such a situation, one may use the uniform bound on $S$ and a $z$-dependent bound on $\fnorm{g_z'}$ as well as particular properties of $W$ to prove accurate bounds on the Kolmogorov distance.
This was done in \cite{CFR11} for the exponential distribution. Incidentally, in the case of the Beta distribution, the function $S$ will be bounded for a different purpose in the proof of Proposition \ref{betabounds}. 
\end{enumerate}
\end{remark}

Proposition \ref{genprop3} is already sufficient to prove that the operator $L$ given by \eqref{gensteinop} characterizes 
the distribution of $Z$. The proof is given in Section \ref{proofs}.

\begin{prop}\label{steincharprop}
A random variable $X$ with values in $\abquer$ has the same distribution as $Z$ if and only if for each continuous function $f$ on $\abquer$, which is locally absoulutely continuous on $(a,b)$ and which satisfies \\
$E\abs{\eta(Z)f'(Z)}=\int_a^b\abs{f'(x)}I(x)dx<\infty$ we have 
\[E[\eta(X)f'(X)]=-E[\gamma(X)f(X)]\,.\]
In particular, in this case both expected values exist.
\end{prop}

Next, we will turn to Lipschitz continuous test functions $h$. In contrast to bounded measurable test functions, there we will also be able to prove useful bounds for $g_h'$. 
In order that $E[h(Z)]$ exists for Lipschitz continuous test functions $h$ we need to assume that $E\abs{Z}<\infty$.
The following two result, which are also proved in Section \ref{proofs}, include optimal bounds for both, $g_h$ and $g_h'$, when $h$ is Lipschitz.

\begin{prop}\label{genprop4}
Assume that $E\abs{Z}<\infty$ and Conditions \ref{gencond1} and \ref{gencond2} hold. Then, we have for any Lipschitz continuous test function 
$h:\abquer\rightarrow\R$ and any $x\in\abquer$:
\begin{enumerate}[{\normalfont (a)}]
\item $\displaystyle\abs{g_h(x)}\leq\fnorm{h'}\frac{F(x)E[Z]-\int_a^x yp(y)dy}{I(x)}=\fnorm{h'}\frac{\int_a^x(E[Z]-y)p(t)dt}{I(x)}\,$;
\item $\displaystyle\abs{g_h'(x)}\leq\fnorm{h'}\frac{\int_a^xF(s)dsG(x)+\int_x^b(1-F(s))dsH(x)}{p(x)\eta(x)^2}$.
\end{enumerate}
Here, for $x\in\abquer$, the positive functions $H(x)$ and $G(x)$ are defined by 
\[H(x):=I(x)-\ga(x)F(x)=p(x)\eta(x)-\ga(x)F(x)\text{ and } G(x):=H(x)+\ga(x)\,.\]
Moreover, these bounds are optimal among all bounds involving the factor $\fnorm{h'}$.  
\end{prop}

\begin{remark}\label{genrem3}
\begin{enumerate}[(a)]
\item If $a>-\infty$ and $b<\infty$, then it follows by an application of de l'H\^{o}pital's rule that the function $S(x):=\frac{\int_a^x(E[Z]-y)p(t)dt}{I(x)}$ is bounded on $(a,b)$. Indeed, if $a>-\infty$, for instance, we have that 
\begin{equation*}
0\leq\lim_{x\downarrow a}\frac{\int_a^x(E[Z]-y)p(t)dt}{I(x)}=\lim_{x\downarrow a}\frac{E[Z]-x}{\ga(x)}=\frac{E[Z]-a}{\ga(a)}
<\infty\,.
\end{equation*}
 However, in general $S(x)$ is unbounded, if $\abs{\ga(x)}$ does not grow at least linearly with $x$. For instance, 
if $Z\sim N(0,1)$ and $\ga(t)=-\sign(t)$, then we have for positive $x$ that
\[S(x)=\frac{\phi(x)}{1-\Phi(x)}\sim x\]
by the Gaussian Mills ratio inequality.
\item The bound for $\abs{g_h(x)}$ in part (a) of Proposition \ref{genprop4} can be written as 
\begin{equation*}
\abs{g_h(x)}\leq\fnorm{h'}\frac{\tau(x)}{\eta(x)}\,,
\end{equation*}
where $\tau$ is the so-called \textit{Stein factor} or \textit{Stein kernel} of $Z$ given by 
\begin{equation*}
\tau(x)=\frac{1}{p(x)}\int_a^x\bigl(E[Z]-t\bigr)p(t)dt\,,
\end{equation*}
i.e. $\tau$ is the function $\eta$ which belongs to the choice $\ga(x)=E[Z]-x$. The Stein kernel $\tau$ appeared first in Lecture $6$ of \cite{St86} and it has turned out to be a fundamental object in Stein's method for one-dimensional absolutely continuous distributions (see, e.g. \cite{NouPec09a}, \cite{NouVie09} and \cite{LRS14}).
\end{enumerate}
\end{remark}

\begin{cor}\label{gencor2}
Assume that $E\abs{Z}<\infty$, Condition \ref{gencond1} holds and that $\ga(x)=c(E[Z]-x)$ for some $c>0$. Then we have for any Lipschitz continuous test function $h:\abquer\rightarrow\R$ and each $x\in(a,b):$
\begin{enumerate}[{\normalfont (a)}]
 \item $\displaystyle\fnorm{g_h}\leq \frac{\fnorm{h'}}{c}\,$;
 \item $\displaystyle\abs{g_h'(x)}\leq\frac{2\fnorm{h'}}{c}\frac{H(x)G(x)}{\eta(x)^2p(x)}
=2c\fnorm{h'}\frac{\int_a^xF(s)ds\int_x^b(1-F(t))dt}{\eta(x)^2p(x)}$.
\end{enumerate}
\end{cor}

\begin{remark}\label{genrem4}
\begin{enumerate}[(i)]
 \item In the case of the normal distribution (via its classical Stein equation) the bound given in Corollary \ref{gencor2} (a) reduces to $\fnorm{g_h}\leq\fnorm{h'}$. Formally, this bound is a special instance of a general bound given in Lemma 3.1 of \cite{GolRin96} for the multivariate standard normal distribution (see also Lemma 2.6 in \cite{CGS}). However, this lemma is stated under the additional assumption that $h$ has three bounded derivatives, which is stronger than being Lipschitz-continuous. Yet, as has been pointed out to me by the referee, one can use the generator representation of the solution to the Stein equation to obtain the same bound as in 
 Corollary \ref{gencor2} (a) for once differentiable test functions $h$ with bounded first derivative by applying the well-known consequences of the dominated convergence theorem on differentiating under the integral sign. 
 Then, using smoothing techniques, this result could be extended to the class of Lipschitz-continuous test functions, yielding an alternative proof of this bound.
 Nevertheless, in the context of Stein's method for the univariate normal distribution, the best bound mentioned on $g_h$ for a Lipschitz test function $h$ is $\fnorm{g_h}\leq2\fnorm{h'}$ (see, e.g. \cite{CGS} or \cite{CheShaSing}). Hence, we believe that Corollary \ref{gencor2} (a) is the first result that rigorously proves the aforementioned bound, although, as described above, it can also be proved by means of existing techniques from the generator framework. 
 \item For concrete distributions the ratio appearing in the bounds for $g_h'(x)$ may be bounded uniformly in $x$ by some constant which can sometimes also be computed explicitely. For instance, this is performed for the Beta distribution in Section \ref{Beta}. Furthermore, for the situation of Corollary \ref{gencor2}, in \cite{EdViq} the authors give mild conditions for the existence of a finite constant $k$ such that $\fnorm{g_h'}\leq k\fnorm{h'}$ for any Lipschitz-continuous $h$. In practice, these conditions are usually met. However, there is no hope of estimating the constant $k$ by their method of proof. Thus, for concrete distributions and explicit constants it might therefore by useful to work with our bounds from Corollary \ref{gencor2} (b) or from Proposition \ref{genprop4}.
\item For the normal distribution and also for the larger class of distributions discussed in \cite{EiLo10}, one also has a bound of the form $\fnorm{g_h''}\leq C\fnorm{h'}$ for some finite constant $C$ holding for each Lipschitz function $h$. As was shown by a universal counterexample in \cite{EdViq}, if $\ga(x)=c(E[Z]-x)$ such a bound cannot be expected unless $a=-\infty$ and $b=\infty$. If either $a>-\infty$ or $b<\infty$ one will have to assume that $h'$ is also Lipschitz, for example by demanding that $h$ has two bounded derivatives, in order to obtain a finite bound on $\fnorm{g_h''}$. Within the density approach, however, there are many examples of distributions, whose support is strictly included in $\R$ but for which such bounds are available (see, e.g., chapter 13 of \cite{CGS}).     
\end{enumerate}
\end{remark}

Now, we show how we can use the above results and the density formula \eqref{densform1} to give bounds on higher order 
derivatives of $g_h$, if $h$ itself is smooth enough. First note that the constant $K$ from \eqref{densform1} is given by 
\begin{equation}\label{Kform}
K=\eta(x_0)p(x_0)=\int_a^{x_0}\abs{\gamma(t)}p(t)dt=\int_{x_0}^b\abs{\gamma(t)}p(t)dt=\frac{E\abs{\gamma(Z)}}{2}
\end{equation}
and, hence, we have the explicit formula
\begin{equation}\label{densform2}
p(x)=\frac{E\abs{\gamma(Z)}}{2\eta(x)}\exp\Bigl(\int_{x_0}^x\frac{\gamma(t)}{\eta(t)}dt\Bigr)\,.
\end{equation}
Formula \eqref{densform2} is a more general version of formula (3.14) in \cite{NouVie09} and is also derived in 
\cite{KuTu12}.
Now, if the coefficient $\gamma$ is also absolutely continuous, by differentiating Stein's equation (\ref{gensteineq}), we obtain for $h$ Lipschitz
\begin{equation}\label{gensecder}
\eta(x)g_h''(x)+g_h'(x)\bigl(\eta'(x)+\ga(x)\bigr)=h'(x)-\ga'(x)g_h(x)=:h_2(x)\,. 
\end{equation}
This means, that the function $\tilde{g}:=g_h'$ is a solution of the Stein equation corresponding to the test function $h_2$ for the distribution of $\Ztilde$ which satisfies the Stein identity
\[E\Bigl[\eta(\Ztilde)f'(\Ztilde)+\bigl(\eta'(\Ztilde)+\ga(\Ztilde)\bigr)f(\Ztilde)\Bigr]=0\,.\]
From \eqref{densform2} we know that a density $\tilde{p}$ of $\Ztilde$ is given by
\begin{equation}\label{formelptilde}
 \tilde{p}(x)=\frac{\tilde{K}}{\eta(x)}\exp\Bigl(\int_{x_0}^x\frac{\eta'(t)+\ga(t)}{\eta(t)}dt\Bigr)=K\exp\Bigl(\int_{x_0}^x\frac{\ga(t)}{\eta(t)}dt\Bigr)=C\eta(x)p(x)\,,
\end{equation}
where $\tilde{K},K,C>0$ are suitable normalizing constants. Thus, if we have bounds for the first derivative of the Stein solutions for the distribution of $\Ztilde$ and for Lipschitz functions $h$, then from this observation we obtain bounds on $g_h''$ for $h$ such that $h_2$ is Lipschitz. Note that if $\ga(x)=c(E[Z]-x)$, this essentially means that $h'$ must be Lipschitz as well. Of course, in order to apply this procedure, one has to make sure that $E[h_2(\Ztilde)]=0$ and that $g_h'$ is the standard solution to the Stein equation for $\calL(\Ztilde)$ and the test function $h_2$. Remarkably, under mild conditions this turns out to always be the case. 

\begin{prop}\label{genprop5}
Assume that Conditions \ref{gencond1} and \ref{gencond2} hold, $E\abs{Z}<\infty$,\\
 $E\abs{Z\ga(Z)}<\infty$ and that $\gamma$ is locally absolutely continuous on $(a,b)$. Furthermore, let $h$ be a Lipschitz-continuous function. Then, $E[h_2(\Ztilde)]$ exists and equals $0$. 
Furthermore, if either the derivative $g_h'$ of $g_h$  is bounded, $a>-\infty$ or $b<\infty$, then $g_h'$ is the standard solution to the Stein equation 
\begin{equation}\label{steineqder}
\eta(x) f'(x)+ \bigl(\eta'(x)+\ga(x)\bigr)f(x)=h_2(x)
\end{equation}
corresponding to the distribution of $\Ztilde$ and the test function $h_2=h'-\gamma'g_h$. 
\end{prop}
The proof is given in Section \ref{proofs}.\\

\begin{remark}\label{secderrem}
If $Z\sim Beta(a,b)$, then \eqref{formelptilde} implies that $\Ztilde\sim Beta(a+1,b+1)$. This will be used in Section \ref{Beta} to provide bounds on higher order derivatives of $g_h$ in the case of the Beta distribution. If, on the other hand, $Z\sim \Exp(\alpha)$ has an exponential distribution with mean $\alpha^{-1}$, 
then $\Ztilde\sim Gamma(2,\alpha)$ has an Erlang distribution. Using this fact, Proposition \ref{genprop5} and the general bounds from Corollary \ref{gencor2} applied for both the exponential and the $Gamma(2,\alpha)$ distribution, one can, with some work, derive the following bounds on the standard solution $g_h$ to the 
Stein equation 
\begin{equation*}\label{steineqexp}
 xg'(x)+(1-\al x)g(x)=h(x)-Eh(Z_\al)
\end{equation*}
for the distribution $\Exp(\alpha)$, if $h$ is continuously differentiable on $[0,\infty)$ and both $h$ and $h'$ are Lipschitz:
\begin{equation*}
 \fnorm{g_h}\leq\frac{1}{\alpha}\fnorm{h'},\quad\fnorm{g_h'}\leq\fnorm{h'}\quad\text{and}\quad\fnorm{g_h''}\leq \frac{2\al}{3}\fnorm{h'}+\frac{2}{3}\fnorm{h''}
\end{equation*}
These bounds are better than those derived in \cite{FulRos13} and, additionally, since we do not have to assume that $h'(0)=0$ for the bound on $\fnorm{g_h''}$ to be valid, one term in the bounds of Theorems 1.1 and 1.2 from \cite{FulRos13} would drop off, if instead our bounds were used.
\end{remark}

Next, we introduce the approach of exchangeable pairs satisfying the regression properties \eqref{genreg} and \eqref{secmom} in our general framework. As was observed in \cite{Ro08} for the normal distribution, in case of univariate distributional approximations, one does not need the full strength of exchangeability, but equality in distribution of the random variables $W$ and $W'$ is sufficient. This may allow for a greater choice of admissible couplings in several situations, or at least, relaxes the verification of asserted properties. Thus, let $W,W'$ be real-valued random variables defined on the same probability space such that $W\stackrel{\D}{=}W'$.
We will assume, that the random variables $W$ and $W'$ only have values in an interval $(a,b)\subseteq J\subseteq\abquer$ where both functions $\eta$ and $\ga$ are defined (recall that it might be the case that $\eta$ can only be defined on $(a,b)$). However, from Proposition \ref{etaprop} we know that we can let $J=\abquer$ if 
Condition \ref{gencond3} holds.

\begin{prop}\label{genpluginprop1}
Assume that Conditions \ref{gencond1} and \ref{gencond2} hold and that $W$ is square integrable with $E\abs{\ga(W)}<\infty$. Furthermore, for some constant $\lambda>0$, assume that the general regression property (\ref{genreg}) and also the second moment condition \eqref{secmom} are satisfied by the pair $(W,W')$. 
Let $f:J\rightarrow\R$ be a bounded, continuously differentiable function, which is Lipschitz-continuous and has a Lipschitz-continuous derivative $f'$. Then,
\begin{align}\label{genplugin1}
\Bigl|E\bigl[\eta(W)f'(W)+\ga(W)f(W)\bigr]\Bigr|&\leq\frac{\fnorm{f''}}{6\la}E\bigl[\abs{W'-W}^3\bigr]\nonumber\\
&\;+\fnorm{f}E\abs{R}+\fnorm{f'}E\abs{S}\,,
\end{align}
where $\fnorm{f''}$ denotes the minimum Lipschitz constant of $f'$.
\end{prop}
The proof is given in Section \ref{proofs}.

\begin{remark}\label{plugrem}
\begin{enumerate}[(i)]
\item The bound \eqref{genplugin1} can only be small, if $S$ and $R$ are of negligible order.
\item The proof shows, that Proposition \ref{genpluginprop1} can easily be generalized to the situation, where there is a sub-$\sigma$-algebra $\F$ with 
$\sigma(W)\subseteq\F$ and the more general regression properties 
\begin{equation}\label{genreg2}
\frac{1}{\lambda}E\bigl[W'-W|\F\bigr]=\ga(W)+R\quad\text{and}\quad\frac{1}{2\lambda}E\bigl[(W'-W)^2|\F\bigr]=\eta(W)+S
\end{equation}
hold for some $\F$-measurable remainder terms $R$ and $S$. 
\item If $\calH$ is some class of test functions, such that there are finite, positive constants $c_0$, $c_1$ and $c_2$ with
$\fnorm{g_h}\leq c_0$, $\fnorm{g_h'}\leq c_1$ and $\fnorm{g_h''}\leq c_2$ for each $h\in\calH$, then \eqref{genplugin1} 
immediately yields a bound on the distance
\[d_{\calH}\bigl(\calL(Z),\calL(W)\bigr)=\sup_{h\in\calH}\Bigl|E\bigl[h(W)\bigr]-E\bigl[h(Z)\bigr]\Bigr|\,.\] 
\end{enumerate}
\end{remark}

Finally, in our general framework, we readdress the last issue discussed in Section \ref{motivation}. Namely, we suppose that we are given two functions $\ga$ and $\eta$, such that for some $-\infty\leq a<b\leq\infty$ the function $\ga$ is defined on $\abquer$, $\eta$ is defined at least on $(a,b)$ and the following properties hold. 
\begin{cond}\label{gencond4}
\begin{enumerate}[(a)]
\item The function $\ga$ is decreasing and such that $0<\ga(a+)\leq\infty$ and $-\infty\leq\ga(b-)<0$. Again, we define $x_0\in(a,b)$ by\\ $x_0:=\sup\{x\in(a,b)\,:\,\ga(x)>0\}$. 
\item The function $\eta$ is positive and locally absolutely continuous on $(a,b)$.
\item The function $\ga/\eta$ is locally integrable on $(a,b)$ and, if we define 
\[Q(x):=\int_{x_0}^x\frac{\ga(t)}{\eta(t)}dt\,,\quad x\in(a,b)\,,\]
then we have $Q(a+)=Q(b-)=-\infty$.
\end{enumerate}
\end{cond}

Note that by definition we have $Q(x)\leq 0$ for all $x\in(a,b)$, if Condition \ref{gencond4} is satisfied. Now, we define the density $p$ by relation \eqref{densform1} with $K$ being a suitable normalizing constant. The existence of $K$ follows from the fact that, by Condition \ref{gencond4}, for each $c\in(a,x_0)$ there is a finite constant $L>0$ 
such that $L\ga(x)\geq 1$ for each $x\in(a,c)$. Thus, 
\begin{align}\label{df1}
 \int_a^c\frac{1}{\eta(x)}\exp\left(\int_{x_0}^x\frac{\ga(t)}{\eta(t)}dt\right)dx&\leq L\int_a^c\frac{\ga(x)}{\eta(x)}\exp\left(\int_{x_0}^x\frac{\ga(t)}{\eta(t)}dt\right)dx\notag\\
 &=L\int_a^c Q'(x)\exp(Q(x))dx=L\int_{-\infty}^{Q(c)}e^udu\notag\\
 &=Le^{Q(c)}<\infty\,.
\end{align}
A similar calculation shows that also 
\[\int_d^b \frac{1}{\eta(x)}\exp\left(\int_{x_0}^x\frac{\ga(t)}{\eta(t)}dt\right)dx<\infty\]
for each $d\in(x_0,b)$. Hence, $p$ can be suitably normalized. Now, let $Z$ be a random variable with probability density function $p$. The next result is a generalization of Lemma 3, Lecture 6 in \cite{St86}.

\begin{theorem}\label{etagammatheo}
If Condition \ref{gencond4} is satisfied, then the density $p$ defined by \eqref{densform1} is such that 
\begin{equation*}
E[\ga(Z)]=\int_a^b\ga(t)p(t)dt=0\quad\text{and}\quad \eta(x)=\frac{1}{p(x)}\int_a^x\ga(t)p(t)dt\,,\quad a<x<b\,. 
\end{equation*}
In particular, the theory developed in this section can be applied in this framework.
\end{theorem}

\begin{proof}
Similarly to \eqref{df1} we obtain 
\begin{align*}
\int_a^{x_0}\ga(x)p(x)dx&=K \int_a^{x_0}\frac{\ga(x)}{\eta(x)}\exp\left(\int_{x_0}^x\frac{\ga(t)}{\eta(t)}dt\right)dx\notag\\
&=Ke^{Q(x_0)}=K
\end{align*}
and 
\begin{equation*}
\int_{x_0}^b \ga(x)p(x)dx=-Ke^{Q(x_0)}=-K\,.
\end{equation*}
Thus, $E[\ga(Z)]=0$. The second claim follows from 
\begin{align*}
 \frac{1}{p(x)}\int_a^x\ga(t)p(t)dt&=\frac{1}{p(x)}K\int_a^x\frac{\ga(s)}{\eta(s)}\exp\left(\int_{x_0}^s\frac{\ga(t)}{\eta(t)}dt\right)ds\notag\\
 &=\frac{1}{p(x)}Ke^{Q(x)}=\eta(x)\,,
\end{align*}
since 
\begin{equation*}
 p(x)=\frac{K}{\eta(x)}e^{Q(x)}\,.
\end{equation*}
\end{proof}

\section{Stein's method for the Beta distribution}\label{Beta}
In this section we specialize the theory from Section \ref{abstract} to the family $Beta(a,b)$, $a,b>0$, of Beta distributions as defined in Section \ref{motivation}. Let us fix $a,b>0$ and from now on assume that $Z\sim Beta(a,b)$. Motivated by the P\'{o}lya urn example, the above constructed exchangeable pair $(W,W')$  and by Proposition \ref{polyaprop1} we define the function $\ga:=\ga_{a,b}$ as in Proposition \ref{polyaprop1} and observe that 
\begin{equation*}
E[\ga(Z)]=0\quad\text{since}\quad E[Z]=\frac{a}{a+b}\,.
\end{equation*}
It is thus easy to see that $\ga$ satisfies all assumptions of Condition \ref{gencond2} and also that the Beta density 
$p:=p_{a,b}$ given by \eqref{densbeta} satisfies Conditions \ref{gencond1} and \ref{gencond3}, the latter either directly or by Proposition \ref{propbs}. We claim that the function $\eta$ defined by \eqref{defeta} is given by 
\begin{equation}\label{etabeta}
\eta(x)=x(1-x)\,,\quad x\in[0,1]\,.
\end{equation}
This is equivalent to proving that 
\begin{equation}\label{eta1}
p(x)x(1-x)=\int_0^x\bigl(a-(a+b)t\bigr)p(t)dt\,,\quad 0<x<1\,,
\end{equation}
which easily follows from differentiating both sides of \eqref{eta1} and using \eqref{logdevbeta}. Thus, from Proposition 
\ref{steincharprop} we immediately obtain the following Stein characterization for the Beta distribution. This result substantially extends Theorem 1 in \cite{Sch01} in the case of the Beta distribution, which is weaker as 
it only characterizes the Beta distribution among the class of absolutely continuous distributions with finite second moment.

\begin{prop}\label{charbeta}
A random variable $X$ with values in $[0,1]$ has the distribution $Beta(a,b)$ if and only if for each continuous function 
$f$ on $[0,1]$, which is locally absolutely continuous on $(0,1)$ such that $E\abs{Z(1-Z)f'(Z)}<\infty$, we have 
\begin{equation*}
E\Bigl[X(1-X)f'(X)\Bigr]=(a+b)E\Bigl[\Bigl(X-\frac{a}{a+b}\Bigr)f(X)\Bigr]\,.
\end{equation*}
\end{prop}
 
For the Beta distribution and a mesaurable function $h$ with $E\abs{h(Z)}<\infty$, the Stein equation \eqref{gensteineq} is given by 
\begin{equation}\label{betasteineq}
x(1-x)g'(x)+(a+b)\Bigl(\frac{a}{a+b}-x\Bigr)g(x)=h(x)-E[h(Z)]\,,\quad x\in[0,1]
\end{equation}
and the standard solution \eqref{gensteinsol} has the form 
\begin{equation}\label{betasteinsol}
g_h(x)=\frac{1}{x(1-x)p(x)}\int_0^x\htilde(t)p(t)dt=\frac{-1}{x(1-x)p(x)}\int_x^1\htilde(t)p(t)dt\,,\quad 0<x<1
\end{equation}
where, again, $\htilde(t)=h(t)-E[h(Z)]$ and 
\begin{equation}\label{betasteinsol2}
g_h(0)=\frac{h(0+)-E[h(Z)]}{a}\quad\text{and}\quad g_h(1)=\frac{h(1-)-E[h(Z)]}{-b}\,,
\end{equation}
if $h$ has a right limit at $0$ and a left limit at $1$ by Proposition \ref{genprop2}. We mention that the same Stein equation \eqref{betasteineq} has already been considered in \cite{Sch01}, \cite{GolRei13} and in \cite{Doe12c}.

For $a,b>0$ define the constant
\begin{align}
 C(a,b)&=2(a+b)
\begin{cases}
 B(a,b),&a\leq 1,\;b\leq1\\
 a^{-1},&a\leq1,\;b>1\\
 b^{-1},&a>1,\;b\leq1\\
 a^{-1}b^{-1}B(a,b)^{-1},&a>1,\;b>1 
\end{cases}\quad\text{if}\quad a\not=b\quad\text{and}\label{lipcon1}\\
 C(a,a)&=\begin{cases}
  4,&0<a<1\\
  \frac{2a\sqrt{\pi}\Gamma(a)}{\Gamma(a+1/2)},&a\geq1.
 \end{cases}\label{lipcon2}
\end{align}
From Proposition \ref{genprop3} and Corollary \ref{gencor2} we can derive the following bounds for the solution \eqref{betasteinsol} 
to \eqref{betasteineq}. The proof is given in Section \ref{proofs}.
\begin{prop}\label{betabounds}
Let $h:[0,1]\rightarrow\R$ be Borel-measurable with $E\abs{h(Z)}<\infty$.
\begin{enumerate}[{\normalfont(a)}]
\item If $h$ is bounded, then $\displaystyle\fnorm{g_h}\leq\frac{\fnorm{h-E[h(Z)]}}{2m(1-m)p(m)}$, where $m$ is the median of $Beta(a,b)$.
\item If $h$ is Lipschitz, then $\displaystyle\fnorm{g_h}\leq\frac{\fnorm{h'}}{a+b}$ and $\fnorm{g_h'}\leq C(a,b)\fnorm{h'}$, where 
$C(a,b)$ is given by \eqref{lipcon1} and \eqref{lipcon2}.
\item If $h$ is continuously differentiable with Lipschitz derivative $h'$, then $g_h'$ is Lipschitz and 
$\displaystyle\fnorm{g_h''}\leq C(a+1,b+1)\fnorm{h''}+(a+b) C(a+1,b+1)C(a,b)\fnorm{h'}$.
\item More generally, if $m\geq1$ is an integer and $h$ is at least $(m-1)$-times differentiable such that $h^{(j)}$ is Lipschitz-continuous for $j=0,\dotsc,m-1$, then $\fnorm{g_h^{(m-1)}}$ is Lipschitz and 
\begin{align*}
 \fnorm{g_h^{(m)}}&\leq C(a+m-1,b+m-1)\\
 &\;\cdot\sum_{j=1}^m\Biggl(\prod_{l=j}^{m-1}\bigl(l(a+b+l-1)C(a+l-1,b+l-1)\bigr)\Biggr)\fnorm{h^{(j)}}\,,
\end{align*}

where we define an empty product to be equal to $1$.
\end{enumerate}
\end{prop}

\begin{remark}\label{betaboundsrem}
\begin{enumerate}[(i)]
 \item It is worthwhile to compare our bound for $\fnorm{g_h'}$ from Proposition \ref{betabounds} (b) to the bound $\fnorm{g_h'}\leq(b_0+b_1)\fnorm{h'}$ given in \cite{GolRei13}. One can show that if $a=b$, then our bound is uniformly better than theirs.
 However, if $a\not=b$, then there are regions for $(a,b)$ where our constant $C(a,b)$ is smaller and other ones, where their $b_0+b_1$ is smaller. For instance, if $0<a,b\leq1$, then, again, $C(a,b)\leq b_0+b_1$. But, if $1<b<2$ is fixed and $a$ tends to zero, then $C(a,b)$ goes to infinity 
 while their $b_0+b_1$ tends to $12$. In any case, neither our bound nor the bound from \cite{GolRei13} seem to be optimal for $\fnorm{g_h'}$.
 \item Form Corollary \ref{gencor2} (b) we know that for Lipschitz $h$ and $x\in(0,1)$
 \[\abs{g_h'(x)}\leq\frac{2\fnorm{h'}}{a+b}\frac{H(x)G(x)}{x^2(1-x)^2p(x)}=:\fnorm{h'}B(x)\,.\]
 By an application of de l'H\^{o}pital's rule, one can show that 
 \[B(0+)=\frac{2}{a+1}\quad\text{and}\quad B(1-)=\frac{2}{b+1}\,.\]
 We conjecture that if $\min(a,b)<1$, then 
 \[\fnorm{B}=\frac{2}{\min(a,b)+1}\,,\]
 i.e. that $B$ assumes its maximum value at the boundary of $(0,1)$. However, if $\min(a,b)>1$, then we believe that there is always an $x_1\in(0,1)$ such that 
 \[\frac{2\fnorm{h'}}{a+b}\frac{H(x_1)G(x_1)}{x_1^2(1-x_1)^2p(x_1)}>\frac{2\fnorm{h'}}{\min(a,b)+1}\,.\]
\item If $a=b$, then the median of $Beta(a,a)$ equals $1/2$ and the bound in (a) has the explicit form $\fnorm{g_h}\leq B(a,a)2^{a+b-1}\fnorm{h-E[h(Z)]}$. Unfortunately, for $a\not=b$ there is no 
closed from expression for the median of $Beta(a,b)$. In such a case one could use known inequalities about the median $m$ in order to get bounds on $\fnorm{g_h}$. Since one would have to distinguish several cases 
according to the values of $a$ and $b$ and, hence, to the shape of the density $p$, we omit the details, here. 
 \end{enumerate}
 
\end{remark}

From Proposition \ref{genpluginprop1}, Remark \ref{plugrem} (ii) and the bounds from Proposition \ref{betabounds} we obtain the following plug-in result, which bounds a certain distance to the Beta distribution by terms related to a given exchangeable pair.

\begin{theorem}\label{betaplugin}
Let $W$ and $W'$ be identically distributed random variables on a common probability space $(\Om,\A,P)$ and let $\F\subseteq\A$ be a 
sub-$\sigma$-algebra of $\A$ such that $\sigma(W)\subseteq\F$ and
\begin{align*}
\frac{1}{\lambda}E\bigl[W'-W\,\bigl|\,\F\bigr]&=(a+b)\left(\frac{a}{a+b}-W\right)+R\quad\text{and}\\
\frac{1}{2\lambda}E\bigl[(W'-W)^2\,\bigl|\,\F\bigr]&=W(1-W)+S
\end{align*}
hold for a constant $\lambda>0$ and for $\F$-measurable remainder terms $R$ and $S$. Then, for each continuously differentiable function 
$h:[0,1]\rightarrow\R$ with a Lipschitz derivative $h'$ it holds that 
\begin{align*}
&\bigl|E[h(W)]-E[h(Z)]\bigr|\leq\fnorm{h'}\left( \frac{1}{a+b}E\abs{R}+C(a,b) E\abs{S}\right)\\
&\,+ \left(\frac{C(a+1,b+1)\fnorm{h''}+(a+b) C(a+1,b+1)C(a,b)\fnorm{h'}}{6\lambda}\right) E\abs{W'-W}^3\,,\\
\end{align*}
where the constants $C(\cdot,\cdot)$ are defined by \eqref{lipcon1} and \eqref{lipcon2}.
\end{theorem} 

Now we are in a position to prove Theorem \ref{mt}.

\begin{proof}[Proof of Theorem \ref{mt}]
The claim immediately follows from Theorem \ref{betaplugin}, Propositions \ref{polyaprop1}, \ref{polyaprop2} and the fact that in this case 
\[R=0,\quad S=\frac{b-a}{2n}W+\frac{a}{2n}\geq0\quad\text{and}\quad  \abs{W'-W}\leq\frac{1}{n}\,.\]
\end{proof}

\section{Proofs}\label{proofs}

\begin{proof}[Proof of Proposition \ref{etaprop}]
Suppose, that $a>-\infty$ and choose $y\in(a,x_0)$. Then $\ga(y)>0$ and, by the nonnegativity of $I$ and the monotonicity of $\ga$, for $a<x<y$ we have
\begin{align}\label{ep1}
0&\leq \ga(y)F(x)=\ga(y)\int_a^x p(t)dt\leq I(x)=\int_a^x\ga(t)p(t)dt\notag\\
&\leq\ga(a)\int_a^x p(t)dt=\ga(a)F(x)\,.
\end{align}
Hence, if Condition \ref{gencond3} holds, we have 
\begin{align*}
0\leq\liminf_{x\downarrow a}\eta(x)\leq\limsup_{x\downarrow a}\eta(x)=\limsup_{x\downarrow a}\frac{I(x)}{p(x)}
\leq\ga(a)\lim_{x\downarrow a}\frac{F(x)}{p(x)}=0\,,
\end{align*}
so that $\lim_{x\downarrow a}\eta(x)=0$. Conversely, if $\eta(a+)=0$, then, again by \eqref{ep1},
\begin{align*}
 0\leq\liminf_{x\downarrow a}\frac{F(x)}{p(x)}\leq\limsup_{x\downarrow a}\frac{F(x)}{p(x)}\leq\frac{1}{\ga(y)}\limsup_{x\downarrow a}\frac{I(x)}{p(x)}=\frac{1}{\ga(y)}\limsup_{x\downarrow a}\eta(x)=0\,.
\end{align*}
The calculation for finite $b$ is similar by using the representation\\ $I(x)=-\int_x^b \ga(t)p(t)dt$ and is therefore omitted. \\ 
\end{proof}

\begin{proof}[Proof of Proposition \ref{propbs}]
That item (a) is sufficient is clear. If (b) holds, then the claim follows from the inequality
\[F(x)=\int_a^x p(t)dt\leq p(x)(x-a)\,,\]
valid for $x\in(a,a+\delta)$. Under Condition (c) we obtain a continuous and convex function on $[a,a+\delta)$ by letting $p(a):=0$. Now, let $a<x<y<a+\delta$. Then, there exists a $\la\in(0,1)$ with $x=\la a+(1-\la)y$ and by convexity we have:
\begin{align*}
p(y)-p(x)&=p(y)-p\bigl(\la a+(1-\la)y\bigr)\geq p(y)-\la p(a)-(1-\la)p(y)\\
&=\la p(y)>0\,.
\end{align*}
Thus, the assumptions of (b) are satisfied. If (d) holds, then again letting $p(a):=0$ we obtain a continuous and concave function on $[a,a+\delta)$. Thus, there exists a decreasing function $f$ on $[a,a+\delta)$ such that 
\begin{equation*}
 p(x)=\int_a^x f(t)dt\,,\quad a\leq x<a+\delta\,.
\end{equation*}
If there was a sequence $(x_n)_{n\geq1}$ in $[a,a+\delta)$ such that $x_n\downarrow a$ and $f(x_n)\leq0$ for each $n\geq1$, then for each $x\in (a,a+\delta)$ and large enough $n$ we would have
\begin{equation*}
 p(x)=p(x_n)+\int_{x_n}^x f(t)dt\leq p(x_n)+(x-x_n)f(x_n)\leq p(x_n)\stackrel{n\to\infty}{\longrightarrow}0\,,
\end{equation*}
which would contradict Condition \ref{gencond1}. Thus, there is an $\eps<\delta$ such that $f(x)\geq0$ for all $x\in(a,a+\eps)$. Hence, $p$ is increasing on $(a,a+\eps)$ and (b) is satisfied.
If (e) holds, then there is an 
$r>0$ and a real sequence $(c_k)_{k\geq0}$ such that $p(x)=\sum_{k=0}^\infty c_k (x-a)^k$ for all $x\in(a,a+r)$ and the function $f:(a-r,a+r)\rightarrow\R$ with $f(x):=\sum_{k=0}^\infty c_k (x-a)^k$ is well-defined. Let $n_0:=\min\{k\geq0\,:\,c_k\not=0\}$. Then $n_0<\infty$ since 
$p$ is positive on $(a,b)$. If $n_0=0$ and hence $f(a)=c_0=\lim_{x\downarrow a}p(x)\not=0$, then there is nothing to show. Otherwise, we have 
\[p(x)=(x-a)^{n_0}\sum_{k=n_0}^\infty c_k(x-a)^{k-n_0}\quad\text{and}\quad p'(x)=(x-a)^{n_0-1}\sum_{k=n_0}^\infty kc_k (x-a)^{k-n_0} \]
for all $x\in(a,a+r)$ and hence, by de l'H\^{o}pital's rule,
\begin{align*}
\lim_{x\downarrow a}\frac{F(x)}{p(x)}&=\lim_{x\downarrow a}\frac{p(x)}{p'(x)}
=\lim_{x\downarrow a}(x-a)\frac{\sum_{k=n_0}^\infty c_k(x-a)^{k-n_0}}{\sum_{k=n_0}^\infty kc_k (x-a)^{k-n_0}}\\
&=\frac{c_{n_0}}{n_0c_{n_0}}\lim_{x\downarrow a}(x-a)=0\,.
\end{align*}
In order to prove (f) we show that always
\begin{equation}\label{li}
 \liminf_{x\downarrow a}\frac{F(x)}{p(x)}=0,
\end{equation}
if $p$ satisfies Condition \ref{gencond1}. To show this, define the function $G(x):=\log F(x)$ for $x\in(a,b)$. Then, $G$ is increasing and continuously differentiable on $(a,b)$ and satisfies $G(a+)=-\infty$ and $G(b-)=0$.
If \eqref{li} did not hold, then
\begin{equation}\label{lsG}
 c:=\limsup_{x\downarrow a}G'(x)=\limsup_{x\downarrow a}\frac{p(x)}{F(x)}<+\infty\,.
\end{equation}
Hence, choosing $\delta>0$ such that $G'(x)\leq c+1$ for all $x\in(a,a+\delta]$ we would obtain
\begin{align*}
G(a+)-G(a+\delta)&=-\lim_{x\downarrow a}\int_x^{a+\delta} G'(t)dt\geq  -\lim_{x\downarrow a}(a+\delta-x)(c+1)=-\delta(c+1)\,,
\end{align*}
which would contradict $G(a+)=-\infty$. \\
\end{proof}

\begin{proof}[Proof of Proposition \ref{genprop3}]
With $\htilde=h-E[h(Z)]$, since $I=\eta\cdot p$, we have for $a<x<b$
\[\abs{g_h(x)}=\frac{\abs{\int_a^x \htilde(t)p(t)dt}}{\abs{p(x)\eta(x)}}=\frac{\abs{\int_a^x \htilde(t)p(t)dt}}{I(x)}\leq\fnorm{\htilde}\frac{F(x)}{I(x)}\,.\]
Let $M:(a,b)\rightarrow\R$ be given by $M(x):=\frac{F(x)}{I(x)}$. By l'H\^{o}pital's rule we have 
\[\lim_{x\searrow a}M(x)=\lim_{x\searrow a}\frac{p(x)}{\ga(x)p(x)}=\lim_{x\searrow a}\frac{1}{\ga(x)}=\frac{1}{\lim_{x\searrow a}\ga(x)}\]
which exists in $[0,\infty)$ by Condition \ref{gencond2}. Here, we used the convention $\frac{1}{\infty}=0$.
Moreover,
\[\lim_{x\nearrow b}M(x)=\frac{1}{\lim_{x\nearrow b}I(x)}=+\infty\]
again by Condition \ref{gencond2} and by Proposition \ref{genprop1}. Furthermore, we have 
\begin{align}\label{Hpos}
M'(x)=\frac{p(x)I(x)-p(x)\ga(x)F(x)}{I(x)^2}=\frac{p(x)}{I(x)^2}\Bigl(I(x)-\ga(x)F(x)\Bigr)\geq0 
\end{align}
for each $x\in(a,b)$ since by the positivity of $p$ and because $\ga$ is decreasing 
\[I(x)=\int_a^x\ga(t)p(t)dt\geq\ga(x)\int_a^xp(t)dt=\ga(x)F(x)\,.\]
Hence, $M$ is increasing and, thus, for each $x\in(a,m]$ we have
\[\abs{g_h(x)}\leq \fnorm{\htilde}\frac{F(m)}{I(m)}=\frac{\fnorm{h-E[h(Z)]}}{2I(m)}\,.\]
The same bound can be proved for $x\in(m,b)$ by using the representation 
\[g_h(x)=-\frac{1}{I(x)}\int_x^b(h(t)-E[h(Z)])p(t)dt\]
and the fact that also $1-F(m)=\frac{1}{2}$. 
\end{proof}

The following two lemmas, which are quite standard in Stein's method, will be needed for the proof of Proposition \ref{genprop4}. For proofs we refer to \cite{Doe12a}, for instance.

\begin{lemma}\label{genlemma1}
 Suppose that $p$ satisfies Condition \ref{gencond1} and that $\int_a^b \abs{x}p(x)dx<\infty$. Then, for each $x\in\abquer$ we have:
\begin{enumerate}[{\normalfont (a)}]
 \item $\int_a^xF(t)dt=xF(x)-\int_a^x s p(s)ds\,$;
 \item $\int_x^b (1-F(t))dt=\int_x^\infty sp(s)ds-x(1-F(x))$.
\end{enumerate}
\end{lemma}

\begin{lemma}\label{liplemma}
Suppose that $p$ satisfies Condition \ref{gencond1} and that $E\abs{Z}=\int_a^b \abs{x}p(x)dx<\infty$. Then, for each Lipschitz function $h$, the following assertions hold true:
\begin{enumerate}[{\normalfont (a)}]
\item For each $y\in\R$ we have\\
 $h(y)-E[h(Z)]=\int_{-\infty}^yF(s)h'(s)ds-\int_y^\infty(1-F(s))h'(s)ds$.  
\item For each $x\in\abquer$ we have\\
$\int_a^x(h(y)-E[h(Z)])p(y)dy=-(1-F(x))\int_a^xF(s)h'(s)ds-F(x)\int_x^b(1-F(s))h'(s)ds$.
\end{enumerate}
\end{lemma}

\begin{proof}[Proof of Proposition \ref{genprop4}]
First, we prove (a). Recall the representation 
\[g_h(x)=\frac{1}{I(x)}\int_a^x(h(y)-E[h(Z)])p(y)dy\,.\]
By Lemmas \ref{liplemma} and \ref{genlemma1} we thus obtain that
\begin{align*}
&\;\abs{I(x)g_h(x)}\\
&=\left|-(1-F(x))\int_a^xF(s)h'(s)ds-F(x)\int_x^b(1-F(s))h'(s)ds\right|\\
&\leq\fnorm{h'}\Biggl((1-F(x))\int_a^x F(s)ds+F(x)\int_x^b(1-F(s))ds\Biggr)\\
&=\fnorm{h'}\Biggl((1-F(x))\Bigl(xF(x)-\int_a^x sp(s)ds\Bigr)+F(x)\Bigl(-x(1-F(x))+\int_x^b sp(s)ds\Bigr)\Biggr)\\
&=\fnorm{h'}\Biggl(-\int_a^x sp(s)ds+F(x)\Bigl(\int_a^x sp(s)ds+\int_x^b sp(s)ds\Bigr)\Biggr)\\
&=\fnorm{h'}\Biggl(F(x)E[Z]-\int_a^x yp(y)dy\Biggr)\,,
\end{align*}
implying (a).\\
Now, we turn to the proof of (b). By Stein's equation (\ref{gensteineq}) we obtain for $x\in(a,b)$
\begin{equation}\label{eqb1}
g_h'(x)=\frac{1}{\eta(x)}\Bigl(\htilde(x)-\ga(x)g_h(x)\Bigr)\,,
\end{equation}
where we have again written $\htilde=h-E[h(Z)]$. Using Lemma \ref{liplemma} again, we obtain 
\begin{align}\label{eqb2}
g_h'(x)&=\frac{1}{\eta(x)}\Biggl(\int_a^xF(s)h'(s)ds\Bigl(1+\frac{\ga(x)(1-F(x))}{\eta(x)p(x)}\Bigr)\nonumber\\
&\,+\int_x^b(1-F(s))h'(s)ds\Bigl(-1+\frac{\ga(x)F(x)}{\eta(x)p(x)}\Bigr)\Biggr)\nonumber\\
&=\int_a^xF(s)h'(s)ds\Bigl(\frac{\eta(x)p(x)+\ga(x)(1-F(x))}{\eta(x)^2p(x)}\Bigr)\nonumber\\
&\;+\int_x^b(1-F(s))h'(s)ds\Bigl(\frac{-\eta(x)p(x)+\ga(x)F(x)}{\eta(x)^2p(x)}\Bigr)\,.
\end{align}

Now, consider the functions $H,G:\abquer\rightarrow\R$ with 
\begin{align*}
 H(x)&=I(x)-\ga(x)F(x)=\eta(x)p(x)-\ga(x)F(x)\quad\text{and}\\
 G(x)&=H(x)+\ga(x)=\eta(x)p(x)+\ga(x)(1-F(x))\,.
\end{align*}
From \eqref{Hpos} we already know that $H$ is nonnegative on $(a,b)$. Similarly we prove the nonnegativity of $G$ on $(a,b)$: Since $p$ is positive and $\ga$ is decreasing, for $x$ in $(a,b)$ we have 
\begin{align*}
G(x)&=I(x)+\ga(x)(1-F(x))=-\int_x^b\ga(t)p(t)dt+\ga(x)(1-F(x))\\
&\geq-\ga(x)(1-F(x))+\ga(x)(1-F(x))=0\,.
\end{align*}
By (\ref{eqb2}) we can thus bound
\begin{align}\label{eqb3}
|g_h'(x)|&\leq\fnorm{h'}\Biggl(\int_a^xF(s)ds\frac{G(x)}{\eta(x)^2p(x)}\nonumber\\
&\;+\int_x^b(1-F(s))ds\frac{H(x)}{\eta(x)^2p(x)}\Biggr)\,,
\end{align}
which reduces to the bound asserted in (b). Optimality of the bound in (a) follows from choosing $h(x)=x$ and observing that the above inequalities are in fact equalities, in this case. To see that also the bound in (b) is optimal, for given $x\in(a,b)$ choose a $1$-Lipschitz function $h$ such that $h'(s)=1$ for all 
$s\in(a,x)$ and $h'(s)=-1$ for all $s\in(x,b)$. Then, from \eqref{eqb2} and the nonnegativity of $H$ and $G$, we see that equality holds in \eqref{eqb3}. \\
\end{proof}

\begin{proof}[Proof of Corollary \ref{gencor2}]
Claim (a) follows from Proposition \ref{genprop4} (a) and the observation that in this case we have 
\[I(x)=\int_a^x\ga(y)p(y)dy=c\int_a^x\bigl(E[Z]-y\bigr)p(y)dy\,.\]
Part (b) follows from Proposition \ref{genprop4} (b) and Lemma \ref{genlemma1} by observing that in this case
\begin{align*}
H(x)&=I(x)-\ga(x)F(x)=c\Biggl(\int_a^x\bigl(E[Z]-t\bigr)p(t)dt-\bigl(E[Z]-x\bigr)F(x)\Biggl)\\
&=c\Biggl(E[Z]F(x)-\int_a^xtp(t)dt-E[Z]F(x)+xF(x)\Biggr)\\
&=c\int_a^x F(s)ds
\end{align*} 
and, similarly, $G(x)=c\int_x^b(1-F(s))ds$.\\
\end{proof}

\begin{proof}[Proof of Proposition \ref{steincharprop}]
We first prove necessity. Let $f$ be given as in the proposition. First we show that $E\abs{\gamma(Z)f(Z)}<\infty$. 
We have 
\begin{align*}
&\;E\abs{\gamma(Z)f(Z)}=\int_a^{x_0}\gamma(x)\abs{f(x)}p(x)dx-\int_{x_0}^b\gamma(x)\abs{f(x)}p(x)dx\\
&=\int_a^{x_0}\gamma(x)p(x)\left|\int_x^{x_0} f'(t)dt-f(x_0)\right|dx
-\int_{x_0}^b\gamma(x)p(x)\left|\int_{x_0}^{x} f'(t)dt+f(x_0)\right|dx\\
&\leq\abs{f(x_0)}\left(\int_a^{x_0}\gamma(x)p(x)dx-\int_{x_0}^{b}\gamma(x)p(x)dx\right)
+\int_a^{x_0}\gamma(x)p(x)\int_x^{x_0} \abs{f'(t)}dt\\
&\,-\int_{x_0}^b\gamma(x)p(x)\int_{x_0}^{x} \abs{f'(t)}dt\\
&=\abs{f(x_0)}E\abs{\gamma(Z)}+\int_a^{x_0}\abs{f'(t)}\int_a^t \gamma(x)p(x)dxdt
-\int_{x_0}^b\abs{f'(t)}\int_t^b\gamma(x)p(x)dxdt\\
&=\abs{f(x_0)}E\abs{\gamma(Z)}+\int_a^b\abs{f'(t)} I(t)dt<\infty\,.
\end{align*}
Repeating essentially the same calculation without absolute value signs and using $E[\gamma(Z)]=0$ yields 
\[E[\eta(Z)f'(Z)]=-E[\gamma(Z)f(Z)]\,.\]
To prove sufficiency it is clearly enough to show that 
\[E[h(X)]=E[h(Z)]\]
holds for each bounded and continuous function $h$. Let $g_h$ be the standard solution of the Stein equation 
\eqref{gensteineq} corresponding to $h$. Then, from Proposition \ref{genprop3} we know that $\fnorm{g_h}<\infty$. 
Also, $g_h$ is continuous on $\abquer$ and continuously differentiable on each compact subinterval of $\abquer$. 
Furthermore, since $I(x)=\eta(x)p(x)$ and $g_h$ solves \eqref{gensteineq} we have 
\begin{equation*}
\abs{g_h'(x)}I(x)=p(x)\bigl|\htilde(x)-\gamma(x) g_h(x)\bigr|\leq p(x)\abs{\htilde(x)}+\abs{\gamma(x) g_h(x)}
\end{equation*}
and, hence, 
\begin{equation*}
\int_a^b\abs{g_h'(x)}I(x)dx\leq 2\fnorm{h}+\fnorm{g_h}E\abs{\gamma(Z)}<\infty\,.
\end{equation*}
By the hypothesis of Proposition \ref{steincharprop} we can thus conclude that 
\[0=E\bigl[\eta(X)g_h'(X)+\gamma(X) g_h(X)\bigr]=E[h(X)]-E[h(Z)]\,,\]
as desired.\\
\end{proof}

\begin{proof}[Proof of Proposition \ref{genprop5}]
 We first show that $E[h_2(\Ztilde)]$ exists. Since $\Ztilde$ has density proportional to $\eta p$ by \eqref{formelptilde} and by \eqref{gensecder}, existence follows, if we can show that 
 \begin{equation}\label{sd5}
 \int_a^b\abs{\ga'(x)}\abs{g_h(x)\eta(x)p(x)}dx<\infty\quad\text{and}\quad\int_a^b\abs{h'(x)}\eta(x)p(x)dx<\infty\,. 
 \end{equation}
To show finiteness of the first integral in \eqref{sd5} note that since $\ga$ is decreasing, by Fubini's theorem
\begin{align}\label{sd6}
 &\;\int_a^{x_0}\abs{\ga'(x)}\abs{g_h(x)\eta(x)p(x)}dx=-\int_a^{x_0}\ga'(x)\left|\int_a^x\htilde(t)p(t)dt\right|dx\notag\\
 &\leq-\int_a^{x_0}\ga'(x)\int_a^x\abs{\htilde(t)}p(t)dtdx=\int_a^{x_0}\abs{\htilde(t)}p(t)\int_t^{x_0}\bigl(-\ga'(x)\bigr)dxdt\notag\\
 &=\int_a^{x_0}\abs{\htilde(t)}p(t)\bigl(\ga(t)-\ga(x_0)\bigr)dt=\int_a^{x_0}\abs{\htilde(t)}p(t)\ga(t)dt\notag\\
 &\leq \fnorm{h'}\int_a^{x_0}\abs{t-Z}p(t)\ga(t)dt\leq \fnorm{h'}\Bigl(E\abs{Z}E\abs{\ga(Z)}+ E\abs{Z\ga(Z)}\Bigr)<\infty\,,
\end{align}
since $h$ is Lipschitz. Similarly, one shows that 
\[\int_{x_0}^b\abs{\ga'(x)}\abs{g_h(x)\eta(x)p(x)}dx<\infty\,.\]
 Since $h'$ is bounded, to show that the second integral in \eqref{sd5} is finite, it suffices to prove that 
 \begin{equation}\label{sd7}
 \int_a^b\eta(x)p(x)dx<\infty\,.
 \end{equation}
 We have
 \begin{align*}
  \int_a^{x_0}\eta(x)p(x)dx&=\int_a^{x_0}\int_a^x\ga(t)p(t)dtdx\notag\\
  &=\int_a^{x_0}\ga(t)p(t)\int_t^{x_0}dxdt=\int_a^{x_0}(x_0-t)\ga(t)p(t)dt<\infty\,,
 \end{align*}
since $E\abs{\ga(Z)}<\infty$ and $E\abs{Z\ga(Z)}<\infty$ and similarly one shows that 
\[\int_{x_0}^b\eta(x)p(x)dx<\infty\,.\]
Hence, \eqref{sd7} holds and $E[h_2(\Ztilde)]$ exists. 
 Now, we prove that $E[h_2(\Ztilde)]=0$. From \eqref{gensecder} and \eqref{formelptilde} we see that this amounts to proving
 \begin{equation}\label{sd1}
  \int_a^b h'(x)\eta(x)p(x)dx=\int_a^b \gamma'(x)g_h(x)\eta(x)p(x)dx\,.
 \end{equation}
Using $\eta p=I$, $I'=\gamma p$ and $I(a+)=I(b-)=0$, from Fubini's theorem we obtain that the left hand side of \eqref{sd1} equals
\begin{align}\label{sd2}
 &\;\int_a^{x_0}h'(x) I(x)dx+\int_{x_0}^b h'(x) I(x)dx\notag\\
 &=\int_a^{x_0}h'(x)\int_a^x I'(t)dtdx-\int_{x_0}^bh'(x)\int_{x}^b I'(t)dtdx\notag\\
 &=\int_a^{x_0}I'(t)\int_t^{x_0}h'(x)dxdt-\int_{x_0}^b I'(t)\int_{x_0}^th'(x)dxdt\notag\\
 &=\int_a^{x_0}I'(t)\bigl(h(x_0)-h(t)\bigr)dt-\int_{x_0}^b I'(t)\bigl(h(t)-h(x_0)\bigr)dt\notag\\
 &=\int_a^b \ga(t) p(t)\bigl(h(x_0)-h(t)\bigr)dt=-\int_a^b h(t)\ga(t)p(t)dt\,.
\end{align}
Similarly, using $\gamma(x_0)=0$, the definition of $g_h$ in \eqref{gensteinsol} and Fubini's theorem again, we have that the right hand side of \eqref{sd1} equals
\begin{align}\label{sd3}
 &\;\int_a^{x_0} \gamma'(x)g_h(x)\eta(x)p(x)dx+\int_{x_0}^b \gamma'(x)g_h(x)\eta(x)p(x)dx\notag\\
 &=\int_a^{x_0}\ga'(x)\int_a^x \htilde(t)p(t)dtdx-\int_{x_0}^b \gamma'(x)\int_t^b\htilde(t)p(t)dtdx\notag\\
 &=\int_a^{x_0}\htilde(t)p(t)\int_t^{x_0}\ga'(x)dxdt-\int_{x_0}^b\htilde(t)p(t)\int_{x_0}^t\ga'(x)dxdt\notag\\
 &=-\int_a^b\htilde(t)p(t)\ga(t)dt=-\int_a^bh(t)\ga(t)p(t)dt\,,
\end{align}
where we have used $E[\ga(Z)]=0$ for the last equality. Thus, from \eqref{sd2} and \eqref{sd3} we conclude that \eqref{sd1} holds.\\
Thus, the standard solution $f=f_{h_2}$ to \eqref{steineqder} is well-defined and given by 
\begin{equation}\label{sd8}
 f(x)=\frac{1}{\eta(x)^2p(x)}\int_a^x h_2(t)\eta(t)p(t)dt=\frac{-1}{\eta(x)^2p(x)}\int_x^b h_2(t)\eta(t)p(t)dt\,,\quad a<x<b\,.
\end{equation}
Hence, 
\begin{equation}\label{sd9}
 \lim_{x\downarrow a}\eta(x)^2p(x)f(x)=0=\lim_{x\uparrow b}\eta(x)^2p(x)f(x)
\end{equation}
by dominated convergence. Furthermore, since $g_h'$ is also a solution to \eqref{steineqder} and the solutions to the corresponding homogeneous equation are exactly the constant multiples 
of $\eta^{-2} p^{-1}$, there is a constant $c\in\R$ such that 
\begin{equation}\label{sd4}
g_h'(x)=f(x)+\frac{c}{\eta(x)^2 p(x)}\,,\quad x\in(a,b)\,. 
\end{equation}
Now, first suppose that $a>-\infty$.
Since $g_h$ solves the Stein equation \eqref{gensteineq} we know that 
\begin{equation}\label{sd10}
 \eta(x)^2p(x)g_h'(x)=I(x)\bigl(\htilde(x)-\ga(x)g_h(x)\bigr)\,.
\end{equation}
As $x\downarrow a>-\infty$, by \eqref{conta1}, the term in brackets converges to $\htilde(a)-\ga(a+)\frac{\htilde(a)}{\ga(a+)}=0$ and since $\lim_{x\to a}I(x)=0$ we conclude from \eqref{sd10} that also 
\begin{equation}\label{sd11}
 \lim_{x\downarrow a}\eta(x)^2p(x)g_h'(x)=0\,.
\end{equation}
Hence, from \eqref{sd9}, \eqref{sd11} and \eqref{sd4} we conclude that $g_h'=f$ is the standard solution to \eqref{steineqder}. Similarly, one obtains this result if $b<\infty$. Finally assume that $g_h'$ is bounded.
Since $\lim_{x\downarrow a}\eta(x)^2p(x)=0$ we conclude from \eqref{sd4} and \eqref{sd9} that 
\begin{equation*}
 0=\lim_{x\downarrow a}\eta(x)^2p(x)g_h'(x)=\lim_{x\downarrow a}\eta(x)^2p(x)f(x)+c=c\,.
\end{equation*}
\end{proof}

\begin{proof}[Proof of Proposition \ref{genpluginprop1}]
Let $x_0$ be defined as above and define the function $G:J\rightarrow\R$ by $G(x):=\int_{x_0}^xf(y)dy$. Then, by Taylor's formula, for each $x,x'\in I$ we have
\begin{align*}
G(x')-G(x)&=G'(x)(x'-x)+\int_x^{x'}(x'-t)G''(t)dt\\
&=f(x)(x'-x)+\int_x^{x'}(x'-t)f'(t)dt\\
&=f(x)(x'-x)+(x'-x)^2\int_0^1(1-s)f'\bigl(x+s(x'-x)\bigr)ds\,.
\end{align*}
Hence, by distributional equality, we obtain
\begin{align*}
0&=E\bigl[G(W')\bigr]-E\bigl[G(W)\bigr]\\
&=E\bigr[f(W)(W'-W)\bigr]+E\Bigl[(W'-W)^2\int_0^1(1-s)f'\bigl((W+s(W'-W)\bigr)ds\Bigr]\\
&=E\Bigr[f(W)E\bigl[W'-W|W\bigr]\Bigr]+E\Bigl[(W'-W)^2\int_0^1(1-s)f'\bigl((W+s(W'-W)\bigr)ds\Bigr]\\
&=\la E\bigl[f(W)\ga(W)\bigr]+\la E\bigl[f(W)R\bigr]\\
&+E\Bigl[(W'-W)^2\int_0^1(1-s)f'\bigl((W+s(W'-W)\bigr)ds\Bigr]\,,
\end{align*}
yielding 
\begin{equation}\label{genplugin1eq1}
E\bigl[f(W)\ga(W)\bigr]=-\frac{1}{\la}E\Bigl[(W'-W)^2\int_0^1(1-s)f'\bigl((W+s(W'-W)\bigr)ds\Bigr]-E\bigl[f(W)R\bigr].
\end{equation}
This immediately implies the identity
\begin{align}\label{genplugin1eq2}
&\;E\bigl[\eta(W)f'(W)+\ga(W)f(W)\bigr]\nonumber\\
&=E\Bigr[f'(W)\bigl(\eta(W)-\frac{1}{2\la}(W'-W)^2\bigr)\Bigr]\nonumber\\
&\;+\frac{1}{\la}E\Bigl[(W'-W)^2\int_0^1(1-s)\Bigl(f'(W)-f'\bigl((W+s(W'-W)\bigr)\Bigr)ds\Bigr]\nonumber\\
&\;-E\bigl[f(W)R\bigr]\,.
\end{align}
Observing that 
\[\Bigl| f'(W)-f'\bigl((W+s(W'-W)\bigr)\Bigr|\leq\fnorm{f''}\, s\abs{W'-W}\]
and $\int_0^1s(1-s)ds=\frac{1}{6}$
the bound \eqref{genplugin1} now easily follows from \eqref{genplugin1eq2} and the properties of $f$.\\
\end{proof}

\begin{proof}[Proof of Proposition \ref{betabounds}]
Claim (a) immediately follows from Proposition \ref{genprop3}. Similarly, the first part of claim (b) immediately follows from Corollary \ref{gencor2} (a). 
For the second part of (b) we note that by Corollary \ref{gencor2} (b) we have for $x\in(0,1)$:
\begin{equation}\label{cb1}
 \abs{g_h'(x)}\leq 2(a+b)\fnorm{h'}\frac{\int_0^x F(t)dt\int_x^1(1-F(s))ds}{x^2(1-x)^2p(x)},\quad x\in(0,1). 
\end{equation}
Since $F$ is increasing and $1-F$ is decreasing, we have 
\begin{equation*}
 \int_0^x F(t)dt\leq x F(x)\quad\text{and}\quad \int_x^1(1-F(s))ds\leq (1-F(x))(1-x)
\end{equation*}
for each $x\in[0,1]$. Plugging this into \eqref{cb1} yields
\begin{equation}\label{cb2}
\abs{g_h'(x)}\leq 2(a+b)S(x)\fnorm{h'}\,,
\end{equation}
where 
\begin{equation*}
S(x)=\frac{F(x)(1-F(x))}{\eta(x)p(x)} =\frac{1}{B(a,b)}\frac{\int_0^xt^{a-1}(1-t)^{b-1}dt\int_x^1s^{a-1}(1-s)^{b-1}ds}{x^a(1-x)^b}.
\end{equation*}
By de l'H\^{o}pital's rule, one can easily show that $S(0+)=a^{-1}$ and $S(1-)=b^{-1}$.
Thus, it suffices to bound $\fnorm{S}$. For general $a,b>0$ we write 
\begin{equation}\label{cb5}
 S(x)=\frac{1}{B(a,b)}f_1(x)f_2(x)\,,
\end{equation}
where 
\begin{align*}
 f_1(x)&:=f_1(x;a,b):=\frac{\int_0^xt^{a-1}(1-t)^{b-1}dt}{x^a}\quad\text{and}\\
f_2(x)&:=f_2(x;a,b):=\frac{\int_x^1t^{a-1}(1-t)^{b-1}dt}{(1-x)^b}\,.
 \end{align*}
For $a\not=b$ we bound the functions $f_1$ and $f_2$ seperately. By de l'H\^{o}pital's rule we have
\begin{equation*}
\lim_{x\downarrow0}f_1(x)=\lim_{x\downarrow0}\frac{x^{a-1}(1-x)^{b-1}}{ax^{a-1}}=\frac{1}{a}.
\end{equation*}
Also, note that 
\begin{equation*}
f_1'(x)=\frac{x^a(1-x)^{b-1}-a \int_0^xt^{a-1}(1-t)^{b-1}dt}{x^{a+1}}=:\frac{N_1(x)}{x^{a+1}}\,.
\end{equation*}
We have $N_1(0+)=0$ and 
\begin{equation*}
 N_1'(x)=(1-b)x^a(1-x)^{b-2}\begin{cases}
                             \geq0\;\forall x\in(0,1),& b\leq1\\
                             <0\;\forall x\in(0,1),& b>1\,.
                             \end{cases}
\end{equation*}
This implies that $N_1$ is nonnegative and, hence, $f_1$ is increasing for $b\leq1$ and that $N_1$ is nonpositive and, hence, $f_1$ is decreasing for $b>1$. Thus, 
\begin{equation}\label{cb3}
 \fnorm{f_1}=\begin{cases}
              f_1(1-)=B(a,b),&b\leq 1\\
              f_1(0+)=a^{-1},&b>1\,.
             \end{cases}
\end{equation}
Since $f_2(x;a,b)=f_1(1-x;b,a)$ we have 
\begin{equation}\label{cb4}
 \fnorm{f_2}=\begin{cases}
              B(b,a)=B(a,b),&a\leq1\\
              b^{-1},&a>1\,.
             \end{cases}
\end{equation}
Thus, from \eqref{cb2}, \eqref{cb5}, \eqref{cb3} and \eqref{cb4} we have 
\begin{equation*}
 \fnorm{g_h'}\leq C(a,b)\fnorm{h'},
\end{equation*}
where $C(a,b)$ is given by \eqref{lipcon1} and \eqref{lipcon2}.
In the case $a=b$ we can provide better bounds. First note that in this case the Beta distribution $Beta(a,a)$ is symmetric with respect to $1/2$. This easily implies that 
\begin{equation*}
 S(1/2-x)=S(1/2+x)
\end{equation*}
holds for each $0\leq x\leq 1/2$. Thus it suffices to bound $S$ on $[1/2,1)$. Note that 
\begin{equation*}
S(1/2)=\frac{F(1/2)(1-F(1/2))}{1/2(1-1/2)p(1/2)}=\frac{1}{p(1/2)}=2^{2a-2}B(a,a)
\end{equation*}
and 
\begin{equation*}
S(1-)=F(1)\lim_{x\uparrow1}\frac{1-F(x)}{\eta(x)p(x)}= \lim_{x\uparrow1}\frac{-p(x)}{\gamma(x)p(x)}=\frac{-1}{\gamma(1)}=\frac{1}{b}=\frac{1}{a}\,.
\end{equation*}
For $x\in(0,1)$ we have 
\begin{equation}\label{cb9}
 S'(x)=\frac{T(x)}{\eta(x)^2 p(x)},
\end{equation}
where
\begin{equation}\label{cb6}
T(x)=\eta(x)p(x)(1-2F(x))-F(x)(1-F(x))\gamma(x)\,.        
\end{equation}
Thus, $S$ is increasing (decreasing) on $[1/2,1)$, if and only if $T$ is nonnegative (nonpositive) there. 
In the case $a=b$ we have $\ga(x)=a(1-2x)$ and, hence, $\ga(1/2)=F(1/2)=0$. Thus, recalling that $I(1)=(\eta p)(1-)=0$ we have 
\begin{equation}\label{cb7}
 T(1/2)=T(1-)=0\,.
\end{equation}
By \eqref{cb7} the nonnegativity (nonpositivity) of $T$ on $[1/2,1)$ follows, if $T(y)\geq 0$ ($\leq0$) for every locally extremal point $y\in (1/2,1)$. 
We have 
\begin{equation}\label{cb11}
T'(x)=-2\eta(x)p(x)^2-\gamma'(x)F(x)(1-F(x)) 
\end{equation}
and, hence, if $y\in (1/2,1)$ is a locally extremal point of $T$, we have $T'(y)=0$ and
\begin{align}\label{cb8}
T(y)&=\eta(y)p(y)\Bigl(1-2F(y)+2\frac{\ga(y)}{\ga'(y)}p(y)\Bigr)\notag\\
&=\eta(y)p(y)\Bigl(1-2F(y)+(2y-1)p(y)\Bigr)\,.
\end{align}
Now, for $x\in[1/2,1)$, consider the function
\begin{equation*}
 U(x)=1-2F(x)+(2x-1)p(x)
 \end{equation*}
 and note that $U(1/2)=0$. For $1/2\leq x<1$ we have 
 \begin{equation*}
  U'(x)=(2x-1)p'(x)=p(x)(2x-1)\psi(x)=\frac{p(x)(1-2x)^2}{x(1-x)}(1-a)
 \end{equation*}
 and, hence, $U$ is increasing for $a\leq1$ and is decreasing for $a\geq1$. Since $U(1/2)=0$ it thus follows from \eqref{cb8} that if $y\in(1/2,1)$ is a locally extremal point of $T$, then $T(y)$ is nonnegative for $a<1$ and nonpositive for $a\geq1$. From \eqref{cb9} and \eqref{cb7} it thus follows that 
 $S$ is decreasing on $[1/2,1)$ if $a\geq1$ and increasing if $a<1$. Hence, we can conclude that
 \begin{equation}\label{cb10}
  \fnorm{S}=\sup_{x\in[1/2,1)}S(x)=
  \begin{cases}
  S(1-)=a^{-1},& 0<a<1\\
  S(1/2)=p(1/2)^{-1},&a\geq1\,.
 \end{cases}
 \end{equation}
Note that by the duplication formula for the Gamma function we have  
\begin{align*}
 p(1/2)&=B(a,a)^{-1}\Bigl(\frac{1}{2}\Bigr)^{2a-2}=\Bigl(\frac{1}{2}\Bigr)^{2a-2}\frac{\Gamma(2a)}{\Gamma(a)^2}\\
 &=\frac{2^{2a-1}\Gamma(a+1/2)\Gamma(a)}{\sqrt{\pi}\Gamma(a)^2 2^{2a-2}}=\frac{2\Gamma(a+1/2)}{\sqrt{\pi}\Gamma(a)}.
\end{align*}
Hence, by \eqref{cb2} and \eqref{cb10} this implies 
\begin{equation*}
 \fnorm{g_h'}\leq 4a\fnorm{S}\fnorm{h'}=\fnorm{h'}
 \begin{cases}
  4,&0<a<1\\
  \frac{2a\sqrt{\pi}\Gamma(a)}{\Gamma(a+1/2)},&a\geq1.
 \end{cases}
\end{equation*}
Now, we turn to the proof of (c). From Proposition \ref{genprop5} we know that $g_h'$ is the standard solution to the Stein equation 
\begin{equation*}
 x(1-x)f'(x)+\bigl(a+1-(a+b+2)x\bigr)f(x)=h_2(x)=h'(x)+(a+b)g_h(x)
\end{equation*}
corresponding to the distribution $Beta(a+1,b+1)$. Thus, since $h_2$ is Lipschitz by part (b), applying (b) for $Beta(a+1,b+1)$ and for $Beta(a,b)$ yields 
\begin{align}\label{secder1}
 \fnorm{g_h''}&\leq C(a+1,b+1)\fnorm{h_2'}\leq C(a+1,b+1)\bigl(\fnorm{h''}+(a+b)\fnorm{g_h'}\bigr)\notag\\
 &\leq C(a+1,b+1)\fnorm{h''} + C(a+1,b+1)C(a,b)\fnorm{h'}\,,
\end{align}
as claimed. The proof of (d) is very similar to the proof of (c) which is why we only give a sketch. Defining $h_1:=\htilde$ and for $k=2,\dotsc,m$ 
\begin{equation*}
 h_k(x)=x(1-x)g_h^{(k)}(x)+\bigl(a+k-1-(a+b+2k-2)x\bigr)g_h^{(k-1)}(x)\,,
\end{equation*}
one can see by induction that for all $k=2,\dotsc,m$
\begin{equation*}
 h_k=h^{(k-1)}+(k-1)(a+b+k-2)g_h^{(k-2)}\,.
\end{equation*}
Hence, by (b) and from Proposition \ref{genprop5} similarly to \eqref{secder1} we can prove that 
\begin{align*}\label{mder1}
 &\;\fnorm{g_h^{(m)}}\leq C(a+m-1,b+m-1)\fnorm{h_m'}\notag\\
 &=C(a+m-1,b+m-1)\fnorm{h^{(m)}+(m-1)(a+b+m-2)g_h^{(m-1)}}\notag\\
 &\leq C(a+m-1,b+m-1)\Bigl(\fnorm{h^{(m)}}+(m-1)(a+b+m-2)\fnorm{g_h^{(m-1)}}\Bigr)\,.
\end{align*}
The bound now follows from an easy induction on $m$.
\\
\end{proof}

\normalem
\bibliography{Beta2014}{}
\bibliographystyle{plain}

\end{document}